\documentclass[11pt,a4paper,english]{article}
\usepackage{amsmath, latexsym, amsfonts, amssymb, amsthm, amscd}
\usepackage[hmargin=2.5cm,vmargin=2.5cm]{geometry}
\usepackage{pdfsync}
\usepackage{graphicx}
\usepackage{epsfig}
\usepackage{subfigure}
\usepackage{float}
\usepackage{authblk}
\newtheorem{theorem}{Theorem}
\newtheorem{corollary}{Corollary}
\newtheorem{proposition}{Proposition}
\newtheorem{lemma}{Lemma}

\newtheorem{remark}{Remark}

\newcommand{\R}{\mathbb{R}}

\newcommand{\ud}{\mathrm{d}}

\begin{document}

\title{Branching processes with interactions: sub-critical cooperative regime.}
\author[1]{Adri\'an Gonz\'alez Casanova}
\author[2]{Jos\'e Luis Perez}
\author[3]{Juan Carlos Pardo}

\affil[1]{\footnotesize Instituto de Matem\'aticas, Universidad Nacional Aut\'onoma de M\'exico (UNAM)
\'Area de la Investigaci\'on Cient\'ifica, Circuito exterior, Ciudad Universitaria, 04510, M\'exico, D.F.} 
\affil[2]{\footnotesize Centro de Investigaci\'on en Matem\'aticas A.C., Calle Jalisco s/n. 36240 Guanajuato, M\'exico}
\affil[3]{\footnotesize Centro de Investigaci\'on en Matem\'aticas A.C., Calle Jalisco s/n. 36240 Guanajuato, M\'exico}

\date{}
\maketitle 
\begin{abstract}
In this paper, we introduce a  family of processes with values on the nonnegative integers that describes the dynamics of populations where individuals are allowed to have different types of interactions. The types of interactions that we consider include pairwise interaction, such as competition, annihilation and    cooperation; and 
 interaction among several individuals that can be considered as  catastrophes. We call such families of processes  \textit{branching processes with interactions}. Our aim  is to study their long term behaviour  under a specific regime of the pairwise interaction parameters that we introduce as {\it subcritical cooperative regime}. 
 
Under  such regime, we prove that  a process in this class comes down from infinity and has a  moment dual which turns out to be a jump-diffusion that can be  thought as  the evolution of the frequency of a trait or phenotype  and whose parameters have a classical interpretation in terms of population genetics.
The moment dual is an important tool for characterizing the stationary distribution of branching processes with interactions whenever  such distribution exists but it is an interesting object on its own right.

\end{abstract}

\textbf{Keywords}: 
Branching coalescing processes, coming down from infinity, stability, moment duality, population genetics.

\textbf{MSC}:  Primary 60K35;Secondary 60J80.

 \section{Introduction and main results.}\label{intro}

Branching processes is one of the most important families of probabilistic models that describe the dynamics of a given population. The simplest branching model is the so-called Bienaym\'e-Galton-Watson (BGW) process which is a Markov chain whose time steps are the non overlapping generations  with individuals reproducing independently and  giving birth to a (random) number of offspring in the next generation. These random  offsprings  have all the same probability distribution. 

In the continuous time setting, a similar model can also be introduced.  In this case,  each individual possesses an exponential clock that when it rings the individual dies and is replaced by a random number of offsprings. The number of offsprings and the exponential clock associated to each individual are independent and identically distributed. This model possesses overlapping generations  and  it is known  as BGW process in continuous time. Since we are only interested in  the continuous time setting,  we will refer to them as BGW processes and omit the word continuous time.

In order to make this probabilistic model more realistic, many authors have introduced different types of   density-dependence to branching processes (see for instance Jagers \cite{Jagers},  Lambert \cite{Lambert2005} and the references therein). One approach consists in generalising the birth and death rates of continuous time branching processes by considering polynomial rates as functions of the population size. This way of modelling density dependence seems to be  popular in the biology community (see for instance Mattis and Kiffe \cite{MK} and  N\aa sell \cite{Na}).

In this manuscript, we follow  this  approach by considering polynomial rates as functions of the population size that  can be interpreted as different type of interactions between individuals. To be more precise, we are interested in a model that considers several specific phenomena such as  (pairwise) \textit{competition pressure, annihilation and  cooperation}; and  interaction among several individuals that can be considered as  \textit{catastrophes}. We call this family of processes as {\it branching processes with interactions.}

Before we provide a formal definition of   \textit{branching processes with interactions}, we fix some notation and  recall some examples  that have already appeared in the literature. We  introduce $\mathbb{N}:=\{1,2,\dots\}$, the set of strictly positive integers, $\mathbb{N}_0:=\mathbb{N}\cup\{0\}$ the set of non negative integers,  $\overline{\mathbb{N}}:=\mathbb{N}_0\cup\{\infty\}$ the compactification of $\mathbb{N}_0$ induced by the metric $d(n,m)=|1/n-1/m|$, $\mathbb{R}_+:=(0,\infty)$, the set of strictly positive real valued numbers,  and $\mathbb{R}_{+, 0}:=[0,\infty)$.

Our first example  is the so-called logistic branching process which was deeply studied by Lambert in \cite{Lambert2005}. In this model each  individual produces a random number of offspring  independently of each other,  similarly to  BGW processes, but also considers competition pressure, in other words  each pair of individuals interacts at a fixed rate and one of them is killed as result of this interaction. The logistic branching process $L=(L_t, t\in \mathbb{R}_{+, 0})$,  with positive (i.e. in $\mathbb{R}_{+, 0}$) parameters $c$ and $(\pi_i, i\in \mathbb{N})$ such that $\sum_{i\in \mathbb{N}} \pi_i=\rho> 0$, is a continuous time Markov chain with values in $\overline{\mathbb{N}}:=\mathbb{N}\cup\{0, \infty\}$ with infinitesimal generator $Q=(q_{i,j})_{i, j\in \overline{\mathbb{N}}}$ where
\[
q_{i,j}=\left\{ \begin{array}{ll}
i\pi_{j-i}, & \textrm{ if  } i\in \mathbb{N}\textrm{ and } j>i,\\
di+ci(i-1) & \textrm{ if  } i\in \mathbb{N}\textrm{ and } j=i-1,\\
-i(d+\rho+c(i-1)) & \textrm{ if  } i\in \mathbb{N}\textrm{ and } j=i,\\
0 & \textrm{ otherwise. }
\end{array}
\right .
\]
Observe that $\pi_i/\rho$ represents the probability of having $i$ new individuals born at each reproduction event. It is important to note that the states $\{0,\infty\}$ are absorbing states. Moreover, from  the main results in Lambert \cite{Lambert2005} (see Theorems 2.2 and 2.3), we deduce that  the log-moment condition
\[
\sum_{i\in\mathbb{N}} \log(i)\pi_i<\infty,
\]
is  sufficient  so that  the logistic branching process $L$  does not explode in finite time, almost surely.

It turns out that the logistic branching process is also useful in the field of population genetics. Indeed, it appears in a {\it duality} relationship with the frequency of a phenotype with selective disadvantage in a given population which can be modelled by the following  stochastic differential equation (SDE) with values in $[0,1]$,
$$
X_t=x-\rho\int_0^t X_s(1-X_s)\ud s+\int_0^t\sqrt{X_s(1-X_s)}\ud B_s, 
$$
where $x\in [0,1], \rho> 0$ and $B=(B_t, t\in \R_{+, 0})$ is a standard Brownian motion (see Krone and Neuhauser \cite{KN1, KN2}).  Krone and Neuhauser observed that one can study the above  SDE using the block counting process of the ancestral selection graph which turns out to be a  particular case of the logistic branching process. Namely, if we take $c=1$, $d=0$, $\pi_1=\rho$ and $\pi_i=0$,  for all $i\in\mathbb{N}\setminus\{1\}$, in the logistic branching process $L$ defined before, the moments of $X=(X_t, t\in \mathbb{R}_{+, 0})$ can be written in terms of the moment-generating function of $L$ as follows
$$
\mathbb{E}_x[X_t^n]=\mathbf{E}_n[x^{L_t}], \qquad x\in [0,1], \quad n\in \mathbb{N},\quad t\in \mathbb{R}_{+, 0}, 
$$ 
where $\mathbb{E}_x$ and $\mathbf{E}_n$ denote the expectations of $X$ starting from $x$ and $L$ starting from $n$ individuals, respectively.

The above relationship is known as {\it moment duality}  and appears between  many interesting branching processes with interactions and frequency processes that arise in population genetics. For instance, Athreya and Swart \cite{AS} considered the following moment duality: let $\rho, c,d\in \mathbb{R}_{+, 0}$ and denote by $C=(C_t, t\in \mathbb{R}_{+, 0})$  for the process that counts the number of particles of the branching-coalescing process defined by the initial value $C_0=n$ and the following dynamics; each particle splits into two particles at rate $\rho$, each particle dies at rate $d$ and each ordered pair of particles coalesce into one particle at rate $c$. All these events occur independently of each other. The authors in \cite{AS} called this process as the $(1, \rho,c,d)$-braco-process. Note that the braco-process is also a  particular case of the logistic branching process with parameters $c, d,  \pi_1=\rho$ and $\pi_i=0$ for all $i\in \mathbb{N}\setminus\{1\}$. Its dual process $X=(X_t, t\in\mathbb{R}_{+,0})$ is the unique strong solution taking values in $[0,1]$ of the SDE
\[
 X_t=x_0+\int_0^t (\rho-d)X_s\ud s -\int_0^t\rho X^2_s\ud s +\int_0^t \sqrt{2cX_s(1-X_s)}\ud B_s, \qquad t\in \mathbb{R}_{+,0}, 
\]  
where $x_0\in[0,1]$ and $B$ is a standard Brownian motion. Athreya and Swart called this process the resampling-selection process with selection rate $\rho$, resampling rate $c$ and mutation rate $d$ or shortly the $(1,\rho,c,d)$-resem-process. In particular, they observed the following moment duality
$$
\mathbb{E}_x[(1-X_t)^n]=\mathbf{E}_n[(1-x)^{C_t}], \qquad x\in [0,1], \quad n\in \mathbb{N},\quad t\in \mathbb{R}_{+,0}, 
$$ 
where $\mathbb{E}_x$ and $\mathbf{E}_n$ denote the expectations of $X$ starting from $x$ and $C$ starting from $n$ individuals, respectively. Recently, Alkemper and Hutzenthaler \cite{AH}  derived a unified stochastic picture for the moment duality between the resampling-selection model with the branching coalescing particle process of Athreya and Swart. It is important to note that the previous duality relationships include the moment duality between the Wright-Fisher diffusion and the so-called Kingman's coalescent.

Other type of duality relationships have been considered for haploid population models and two-sex population models by M\"ohle \cite{mohle1} and for the Wright-Fisher diffusions with $d$-types  and the Moran model (both in presence and absence of mutation) by Carinci et al. \cite{CGGR}.

 Due to the power of this relationship, the question of which models allow a moment duality is interesting on its own right. In Section 2.1,   we provide conditions for this moment duality to hold  for a large family of branching processes with interactions that include existing examples in the literature. However, the aim of this manuscript  is the long term behaviour of branching processes with interactions and we use the moment duality technique as a tool to determine the invariant distribution whenever it exists. In particular, we  consider interactions that had appeared independently in the literature, and that had been studied before using moment duality.

As in the examples of above, we are  interested in branching processes where individuals die and reproduce, as in the BGW  process,  but also are allowed to have different types of interactions. To be more precise,  our model has the following dynamics: 
\begin{itemize}
\item[i)] {\it death:} each individual in the population dies at rate $d\in\mathbb{R}_{+,0}$, 
\item[ii)] {\it reproduction:} each individual produces $i$ new individuals at rate $\pi_i\in \mathbb{R}_{+,0}$, for $i\in \mathbb{N}$., 
\item[iii)] {\it competition pressure:} each pair of individuals interact at a fixed rate $c\in \mathbb{R}_{+,0}$ and one of them is killed as result of this interaction (see for instance Athreya and Swart \cite{AS} and Lambert \cite{Lambert2005}),
\item[iv)] {\it annihilation:} each   pair of individuals interact at a fixed rate $a\in \mathbb{R}_{+,0}$ and both of them are killed as result of this interaction (see for instance Athreya and Swart \cite{AS12} and Blath and Kurt \cite{BK11}), and
\item[v)] {\it cooperation:} each pair of individuals interact and produce $i$ new individuals at rate $b_i\in \mathbb{R}_{+,0}$, for $i\in \mathbb{N}$ (see for instance Sturm and Swart \cite{SS15}).
\end{itemize}
Finally, we consider interactions among several individuals in the sense of $\Lambda$-coalescent events as in  Foucart \cite{F13} and Griffiths \cite{G12}  that we call {\it catastrophes}.  Let $\Lambda$ be a finite measure on $[0,1]$. If $n$ individuals are present in the population, each $k-1$-tuple die simultaneously  at rate 
$$
\lambda_{n,k}:=\int_0^1y^{k}(1-y)^{n-k}\frac{\Lambda(\ud y)}{y^2}, \qquad \textrm{for}\quad  k\in\{2,3, \ldots, n\}.
$$
We refer to Pitman \cite{Pitman} and  Sagitov \cite{Sagitov} for a proper definition of $\Lambda$-coalescent events and  processes.

Along the paper, we assume that all the parameters $d, c, a, b_i,  \pi_i,$ for $i\in \mathbb{N}$ are positive (i.e. belong to $\mathbb{R}_{+,0}$) and  $\Lambda$ is a finite measure on $[0,1]$, such that 
\[
0\le \rho:=\sum_{i\in \mathbb{N}} \pi_i<\infty, \qquad 0\le \lambda_i:=\sum_{k= 2}^i \binom{i}{k}\lambda_{i,k} \quad \textrm{ and }\quad 0\le b:=\sum_{i\in \mathbb{N}} b_i<\infty.
\]
Note that when $\rho=0$ (similarly when $b=0$), then $\pi_i=0$ $(b_i=0)$, for all $i\in \mathbb{N}$.

The  branching process with interactions $Z=(Z_t, t\in \mathbb{R}_{+,0})$  is a continuous time Markov chain with values in $\overline{\mathbb{N}}$ whose extended generator $Q=(q_{i,j})_{ i, j\in \overline{\mathbb{N}}}$ is given by
\[
q_{i,j}=\left\{ \begin{array}{ll}
i\pi_{j-i}+i(i-1)b_{j-i}, & \textrm{ if  } i\in\mathbb{N} \textrm{ and } j>i,\\
di+c i(i-1) +\binom{i}{2}\lambda_{i,2}& \textrm{ if  } i\in\mathbb{N}\setminus\{1\}  \textrm{ and } j=i-1,\\
ai(i-1) +\binom{i}{3}\lambda_{i,3}& \textrm{ if  } i\in  \{3,4,\ldots\} \textrm{ and } j=i-2,\\
\binom{i}{k}\lambda_{i,k} & \textrm{ if  }i\in\{4, \ldots\}, k\in\{4, \ldots i\}\textrm{ and } j=i-k+1,\\
-i(d+\rho)-i(i-1)(b+c+a)-\lambda_{i} & \textrm{ if  } i\in \mathbb{N}\setminus\{1\}\textrm{ and } j=i,\\
2a& \textrm{ if  } i=2 \textrm{ and } j=0,\\
d & \textrm{ if  } i=1  \textrm{ and } j=0,\\
-d-\rho & \textrm{ if  } i=1\textrm{ and } j=1,\\
0 & \textrm{ otherwise. }
\end{array}
\right .
\]
Equivalently, the extended generator $Q$ of $Z$ acts as follows. For  $f\in C_0(\overline{\mathbb{N}}, \R)$,
the set of continuous functions from $\overline{\mathbb{N}}$ to $\mathbb{R}$  which vanish at infinity, we have 
\[
Qf(0)=Qf(\infty)=0,\qquad Qf(1)=d\Big(f(0)-f(1)\Big)+\sum_{i\in \mathbb{N}}\pi_i\Big(f(1+i)-f(1)\Big),
\] 
and for all $n\in \mathbb{N}\setminus\{1\},$
\begin{equation}\label{GeneratorLB}
\begin{split}
Qf(n)&=d n\Big(f(n-1)-f(n)\Big)+\sum_{i\in \mathbb{N}}n\pi_i\Big(f(n+i)-f(n)\Big)\\
&+n(n-1)a\Big(f(n-2)-f(n)\Big)+n(n-1)c\Big(f(n-1)-f(n)\Big)\\
&+n(n-1)\sum_{i\in \mathbb{N}} b_i\Big(f(n+i)-f(n)\Big)+\sum_{k=2}^n\binom{n}{k}\lambda_{n,k}\Big(f(n-k+1)-f(n)\Big),
\end{split}
\end{equation}
 whenever it is well-defined. For instance  the domain of the extended generator $Q$, here denoted by $\mathcal{D}(Q)$, includes the set of linear combinations of exponentials functions  of the form 
 \begin{equation}\label{lcexp}
 f(n)=\sum_{\ell=1}^k a_\ell e^{-m_\ell n}
 \end{equation} for $a_1,a_2,\ldots, a_k\in \mathbb{R}$, $m_1, m_2, \ldots, m_k\in\R_+ $, and $k\in\mathbb{N}$, which is dense in $C_0(\overline{\mathbb{N}}, \R)$ by the Stone-Weierstrass Theorem for locally compact sets. Indeed, the latter claim follows from the construction in Section 4.2  and Problem 15 (a) in Section 4.11 in \cite{EK} (we leave the details to the interested reader).

For our purposes, we introduce the following parameters
\[
\mathbf{m}:=-d+\sum_{i\in \mathbb{N}} i\pi_i \qquad \textrm{and}\qquad \varsigma:=-c-2a+\sum_{i\in \mathbb{N}} ib_i,
\]
 that we assume to be finite. As we will see below, we can characterise the long term behaviour of branching processes with interactions depending on the value of what we call the {\it cooperative parameter} $\varsigma\in \mathbb{R}$. We say that a branching process with interactions is {\it supercritical, critical or subcritical cooperative} accordingly as $\varsigma>0$,  $\varsigma=0$ or $\varsigma<0$.

In order to understand the long term behaviour for this family of processes, we  use stability theory for continuous time Markov chains (see for instance Tweedie \cite{Twee}, Chen \cite{Chen}, Meyn and Tweedie \cite{stability0} and the notes of Hairer \cite{Ha}). Our first main result says that a branching process with interactions $Z$ which is subcritical or critical cooperative, i.e. $\varsigma\le 0$,  is conservative or in other words that it does not explode in finite time.  In the particular case when there are no catastrophes and $\varsigma>0$,  then the process $Z$ explodes in finite time with positive probability. 
 
 In the sequel, we denote by $\mathbf{P}_n$  for the law of the process $Z$ starting from $n\in\mathbb{N}_0$. Since $\{0\}$ is an absorbing state, we only consider that $Z_0$ takes values in $\mathbb{N}$, almost surely.
 \begin{theorem}\label{Propfinitness} Let $Z$ be a  branching process with interactions such that $|\mathbf{m}|<\infty$.
If the process is subcritical or critical cooperative, i.e. $\varsigma\le 0$,   then the process  $Z$ is conservative, i.e.
\[
\mathbf{P}_n(Z_t<\infty)=1, \qquad \textrm{for any }\quad n\in \mathbb{N}, \,\,t\in \mathbb{R}_{+,0}.
\]
Moreover if there are no catastrophes, i.e. $\Lambda\equiv 0$,  and $Z$ is supercritical cooperative  (i.e. $\varsigma>0$)   starting at $Z_0\in\{2,3,...\}$ almost surely, then  the process $Z$ explodes in finite time with positive probability.
\end{theorem}  

The proof of the first part of our previous result follows from  a Lyapunov type condition for non-explosion found in Chen \cite{Chen} and  for the supercritical case, we use a coupling argument. The supercritical cooperative regime seems to be more involved under the event of catastrophes. Nonetheless, we believe that there must be cases when the process may explode under the presence of catastrophes. 
 
In the sequel, we assume that the process satisfies $|\mathbf{m}|<\infty$ and that it is subcritical cooperative, i.e $\varsigma<0$, unless we state specifically otherwise. It is important to note that  when there is no annihilation, i.e. $a=0,$ the previous assumption implies that $c>0$.

Our next result characterises the long term behaviour of $Z$. For simplicity, we only deal with the case with no annihilation since  the annihilation case is more involved and will be studied at the end of this section.

\begin{proposition}  \label{FLyap} Assume  $|\mathbf{m}|<\infty$ and that $Z$  is a subcritical cooperative branching process with interactions, i.e. $\varsigma<0$,  such that $\mathbb{P}(Z_0\in\mathbb{N})=1$. If there is no annihilation, i.e. $a=0$, we have
\begin{itemize}
\item[i)] if  $d=0$ and $\rho=0$, then  $Z$ is absorbed in the state $\{1\}$ and $\{0\}$  is not accessible,
\item[ii)] if $d=0$ and $\rho>0$,   then $\{0\}$ is not accessible and $Z$ is positive recurrent in $\mathbb{N}$,
\item[iii)] if $d>0$, then $Z$ is absorbed in the state $\{0\}$.
\end{itemize}
\end{proposition}

The proof of part (ii) of  our previous result follows from   a Foster-Lyapunov conditions for positive recurrence and the remaining cases follows from  a coupling argument. Recurrence in the critical cooperative case does not seems easy to handle.  Actually a different approach than  the one we present here for the subcritical cooperative case is needed and we conjecture that the criteria for recurrence not only depends  on  the cooperative parameters but also on the branching parameters.   Further developments on the critical case  appear in Gonz\'alez-Casanova et al. \cite{GCP}  where a particular example is treated.

 When $a=0=d$ and $\rho>0$, the state $\{0\}$ is not accesible and since $Z$ is irreducible (as a process with values in $\mathbb{N}$) and positive recurrent then there exist a unique stationary distribution (see for instance Theorem 21.14 in \cite{LPW}). Thus a natural question arises: {\it can we determine the invariant distribution of $Z$?} In order to provide a positive answer to this question, we first introduce the moment dual of $Z$ which also exist in the critical cooperative case.

The unique moment dual of the branching process  with interactions $Z$ is a jump-diffusion taking values in $[0,1]$ that can be defined as the unique strong solution of the following stochastic differential equation (SDE for short)
\begin{equation}\label{eq:dualnaintro}
\ud X_t=\mu(X_{t})\ud t+\sigma(X_t)\ud B_t +\int_{(0,1]}\int_{(0,1]}z\Big(\mathbf{1}_{\{u\le X_{t-} \}}-X_{t-}\Big) \widetilde{N}(\ud t,\ud z,\ud u),
\end{equation}
with initial condition $X_0=x\in[0,1]$ and where $\widetilde{N}$ is a compensated  Poisson random measure on  $\mathbb{R}_{+,0}\times (0,1]^2$ with intensity $\ud t z^{-2}\Lambda(\ud z) \ud u$, the functions $\mu:[0,1]\rightarrow \R$ and $\sigma: [0,1]\rightarrow \mathbb{R}_{0, +}$ are continuous and  satisfy
\[
\mu(z):=d(1-z)+\sum_{i\in \mathbb{N}}\pi_{i}(z^{i+1}-z)
\quad\textrm{and}\quad \sigma(z):=\sqrt{2c(z-z^2)+2\sum_{i\in \mathbb{N}} b_i(z^{i+2}-z^2)}.
\]
We denote by $\mathbb{P}_x$, for the law of the process $X=(X_t, t\in \mathbb{R}_{0,+})$ starting from $x\in [0,1]$.

\begin{theorem}\label{lemma:characterization} Assume that $|\mathbf{m}|< \infty$, $a=0$ and $\varsigma \le 0$, then the SDE (\ref{eq:dualnaintro}) with starting point $X_0\in[0,1]$ has a unique strong solution, that we denote by $X$,  taking values on $[0,1]$. Moreover $X$  is the unique (in distribution) moment dual of $Z$, the (sub)critical cooperative  branching process with  interactions  having the same parameters as $X$. More precisely, for $x\in [0,1]$ and $n\in\mathbb{N}_0$, we have 
$$
\mathbb{E}_x[X_t^n]=\mathbf{E}_n[x^{Z_t}]\qquad\text{for}\quad  t\in \mathbb{R}_{+,0}. 
$$
\end{theorem}
\begin{remark}
 Recall that $a=0$. The assumption $\varsigma \le 0$,  guarantees that 
 \[
 c(x-x^2)+\sum_{i\in \mathbb{N}} b_i(x^{i+2}-x^2)\geq 0 \qquad \textrm{for all}\quad  x\in[0,1],
 \] 
 and thus $\sigma:[0,1]\mapsto \R_{+,0}$. Indeed, the latter follows since  \[
c(x-x^2)+\sum_{i\in \mathbb{N}} b_i(x^{i+2}-x^2)= x(1-x)\left(c-\sum_{i\in \mathbb{N}} b_i\sum_{j=0}^{i-1} x^{j+1}\right) ,
 \]
implying that the previous quantity is always positive for any $x\in[0,1]$, if the mapping 
\[
x\in [0,1]\mapsto g(x):=c-\sum_{i\in \mathbb{N}} b_i\sum_{j=0}^{i-1} x^{j+1},
\]
 is positive. On the other hand, the function $g$ is decreasing on $[0,1]$ and satisfies $g(0)=c$.  Under the assumption that $\varsigma< 0$, we have that $c> 0$ and $g(1)=c-\sum_{i\in \mathbb{N}} ib_i=-\varsigma> 0$. If $
\varsigma=0$, we have that $c\ge  0$ (if $c=0$ then $b=0$) and $g(1)=c-\sum_{i\in \mathbb{N}}ib_i= 0$.  Moreover, the assumption that $X_0\in[0,1]$ guarantee  that the  SDE \eqref{eq:dualnaintro} take values in $[0,1]$.  
 \end{remark}
The process $X$ can be thought as  the evolution of the frequency of a trait or phenotype and some of the parameters have a classical interpretation in terms of population genetics. For instance,  $d$ represents  the rate at which a mutation affects an individual,  $\pi_1$ has been interpreted as  the weak selection parameter,  $c$ is also known as the strength of the random genetic drift and the Poisson random measure $N$ may model the occurrence of reproduction events that affect large fractions of the population. Recently, the parameters $(\pi_i, i\ge 2)$,  have been interpreted in terms of frequency dependent selection in Gonzalez-Casanova and Span\`o \cite{GS}.

The interpretation of the parameters $(b_i, i\in \mathbb{N})$ are not so classic. In  Gonzalez-Casanova et al. \cite{GCP} a biological interpretation of the parameter $b_1$ is studied and it  can be  related to the efficiency of  individuals. To be more precise, let us imagine a Wright-Fisher model where the population size is coupled with the frequency of individuals of a given type in such a way that its frequency process  $X^{(N)}=(X_n^{(N)}, n\in \mathbb{N}_0)$ is a Markov chain taking  values in $[0,1]$ such that conditionally on $X_{n-1}^{(N)}=x\in[0,1]$, the r.v. $\lfloor N_x\rfloor X_n^N$ is distributed as a Binomial r.v. with parameters $\lfloor N_x\rfloor$ and $x$, where
$$
N_x=\frac{N}{1-b_1 x},
$$
and $x\mapsto \lfloor x \rfloor$ denotes the floor function.
The interpretation of $N_x$ is related to efficiency, indeed assume that at each generation there are $N$ units of resources available and consider that there are two types of individuals: the {\it inefficient} population that consists of individuals that need $1$ unit of resources to be created and the {\it efficient} population that requires $1-b_1$ units of resources. Since, in a given generation,  the frequency of efficient individuals is $x$, the number of individuals that the next generation will be able to sustain is approximately  $N_x$.
In other words, we can interpret the process $X^{(N)}$  as the frequency of efficient individuals. Now observe that for all $x\in[0,1]$, we have
\[
\mathbb{E}\Big[X_n^{(N)}\Big| X_{n-1}^{(N)}=x\Big]=x,\qquad \mathbb{V}\mathrm{ar}\Big[X_n^{(N)} \Big| X_{n-1}^{(N)}=x\Big]=\frac{x(1-x)}{\lfloor N_x\rfloor}\sim \frac{x(1-x)(1-b_1x)}{N},
\]
as $N$ increases. From  the previous observation, it is not so difficult to guess that  the scaling limit of $(X^{(N)}_{\lfloor Nt \rfloor}, t\ge 0)$ when $N$ goes to infinity, must be the unique strong solution of the following SDE
\[
\ud X_t=\sqrt{2X_t(1-X_t)(1-b_1X_t)}\ud B_t,
\]
which belongs to the family of  diffusions  described by the SDE \eqref{eq:dualnaintro} when $c=1$, $b_1\in(0,1)$ and $\Lambda\equiv 0$. In other words, the interpretation of the parameters $(b_i, i\in \mathbb{N})$ can then be understood as frequency dependent effective population size, in the sense of Gonzalez-Casanova and Span\`o \cite{GS}.

It is important to note that recently, Foucart \cite{F13} and Griffiths \cite{G12} studied asymptotic properties of the so-called $\Lambda$-Wright-Fisher  process with selection  using the moment duality.  The latter process can be defined as the unique strong solution of the  SDE (\ref{eq:dualnaintro}) with $\mu(x)=\rho (x^2-x)$ and $\sigma(x)=0$ and its unique moment dual is a  binary branching process with $\Lambda$-catastrophes. Foucart and Griffiths were interested in understanding under which conditions does the solution of such SDE eventually goes to zero with probability one. This question can be interpreted in a biological sense  as follows: {\it under which condition fixation of the fittest phenotype is certain, in a population with skewed reproduction?}

As we mention before, the moment duality property has been used recently by  Gonzalez-Casanova and Span\`o \cite{GS}  for different purposes. The authors in \cite{GS} studied a model related to selection which happens to be moment dual to the solution of the SDE  (\ref{eq:dualnaintro}) with $c,d =0$ and $b_i=0$, for $i\in \mathbb{N}$ and  used the moment duality to understand under which conditions  fixation of the fitness phenotype is certain. It is important to note that the results that we will present here are complementary to those obtained  in Foucart \cite{F13}, Griffiths \cite{G12} and Gonzalez-Casanova and Span\`o \cite{GS}.  As we will see below fixation at 1 of the process $X$ is related to determining the invariant distribution of $Z$.

One of our main results shows that  when there is no annihilation and $d=0$ and $\rho>0$, the invariant distribution of $Z$ can be determined by the fixation at 1 of its moment dual $X$. 

\begin{theorem}\label{thm:criteria} Assume that there is no annihilation, i.e. $a=0=d$, $|\mathbf{m}|<\infty$ and $\rho>0$. Then   a subcritical cooperative branching process with  interactions  $Z$, i.e. $\varsigma< 0$
has a unique stationary distribution here denoted by $\nu$. Moreover, \begin{equation}\label{eq:stationaryuniform}
\lim_{t\rightarrow \infty} \sup_{n\in \mathbb{N}}||\mathbf{P}_n(Z_t=\cdot)-\nu(\cdot)||=0.
\end{equation}
and if  $X$ is the unique  moment dual of $Z$, then 
$\mathbb{P}_x(X_\infty\in\{0,1\})=1$ for $x\in (0,1)$, where $X_\infty:=\lim_{t\to\infty}X_t$ and the generating function associated to $\nu$ satisfies 
$$
f(x):=\sum_{n\in \mathbb{N}} x^n\nu(n)=\mathbb{P}_x(X_{\infty}=1).
$$
\end{theorem}
It is important to note that when there are no catastrophes, i.e. $\Lambda\equiv 0$, the invariant distribution of $Z$ can be determined explicitly in terms of the scale function of the diffusion  described by the SDE \eqref{eq:dualnaintro} as it is explained below.

\begin{corollary}\label{scale}
Assume that $|\mathbf{m}|<\infty$, $\varsigma< 0$ and that there is no annihilation neither catastrophes (i.e. $a=0$ and $\Lambda\equiv 0$). If $d=0$ and $\rho>0$, we have
\[
 \sum_{n\in \mathbb{N}} x^n\nu(n)=\mathbb{P}_x(X_{T_{0,1}}=1)=\frac{S(x)-S(0)}{S(1)-S(0)}, \qquad \textrm{for}\quad x\in[0,1],
\]
 where $T_{0,1}=\inf\{t\ge 0: X_t\in\{0,1\}\}$ and $S$ denotes the so-called scale function of  $X$ which satisfies
\[
S(x)=\int_{\theta_1}^x \exp\left\{-\int_{\theta_2}^y\frac{2\mu(z)}{\sigma^2(z)}\ud z\right\}\ud y,
\]
 where $\theta_1$ and $\theta_2$  are arbitrary numbers in $[0,1]$.
\end{corollary}
The previous Corollary follows from Theorem \ref{thm:criteria}, the theory of exit problems for diffusions (see for instance  Revuz and Yor \cite{RevuzYor}) and Feller's boundary test (see for instance Proposition 2.4 in \cite{MiUr}).  Indeed, from  Feller's boundary test,  the process is absorbed at  1 (respectively at 0) at finite time with positive probability if $S(1)<\infty$ ($|S(0)|<\infty$) and 
 \[
 \int^1 \frac{S(1)-S(u)}{\sigma^2(u)S^\prime(u)}\ud u<\infty \qquad \left(\int_0 \frac{S(u)-S(0)}{\sigma^2(u)S^\prime(u)}\ud u<\infty
\right). \]
Both conditions at 1 and at 0 are fulfilled in our case since
\[
\frac{2\mu(z)}{\sigma^2(z)}\sim -\frac{\rho}{c} \quad \textrm{as} \quad z\to 0,\qquad\frac{2\mu(z)}{\sigma^2(z)}\sim \frac{\mathbf{m}}{\varsigma} \quad \textrm{as} \quad z\to 1,
\]
and
\[
\frac{S(1)-S(u)}{\sigma^2(u)S^\prime(u)}\sim-\frac{1}{\varsigma} \qquad \textrm{as} \quad u\to 1,\qquad \frac{S(u)-S(0)}{\sigma^2(u)S^\prime(u)}\sim\frac{1}{c} \qquad \textrm{as} \quad u\to 0.
\]
Before we continue with our exposition, we study an interesting example, which is related to the model of efficient individuals studied in  \cite{GCP} and where the invariant distribution of $Z$ can be computed in a closed form. Assume that there are no catastrophes and let $d=0$, $\pi_1=\rho>0$, $b_1=b>0$,  $c>b$ and $\pi_i=0=b_i$, for $i\ge 2$. In this particular case, the unique strong solution $X$ of the  SDE \eqref{eq:dualnaintro} is a diffusion whose parameters are given by
\[
\mu(x):=\rho(x^2-x)
\qquad\textrm{ and }\qquad \sigma^2(x):=2x(1-x)(c-bx).
\]
If $2\rho\ne b$, the scale function associated to the diffusion $X$ satisfies
\[
S(x)=\int_0^x \exp\left\{\frac{\rho}{b}\int_0^y\frac{\ud v}{\frac{c}{b}-v}\right\}\ud y= K\left(1-\left(1-\frac{b}{c}x\right)^{1-\frac{\rho}{b}}\right),\quad  \textrm{for } \quad x\in [0,1],
\]
 where  $K$ is a constant that depends on $(c, b, \rho)$ which is  positive or negative accordingly as $\rho<b$ or $b<\rho$. If $\rho= b$, then 
 \[
 S(x)= K\ln \left(\frac{1}{1-\frac{b}{c}x}\right),\quad  \textrm{for } \quad x\in [0,1],
 \] 
 with $K$ positive.  Therefore,   if $b<\rho$, we have 
 \[
 \sum_{n\ge 0} x^n\nu(n)=\mathbb{P}_x(X_{T_{0,1}}=1)=\frac{S(x)-S(0)}{S(1)-S(0)}=C_{c, b, \rho}\left(\left(1-\frac{b}{c}x\right)^{1-\frac{\rho}{b}}-1\right),
 \]
 where
 \[
 C_{c, b, \rho}=\left(\left(1-\frac{b}{c}\right)^{1-\frac{\rho}{b}}-1\right)^{-1}.
 \]
By the binomial theorem, we deduce 
\[
\nu(k)=C_{c, b, \rho}\frac{\Gamma\Big(\frac{\rho}{b}-1+k\Big)}{\Gamma\Big(\frac{\rho}{b}-1\Big)k!}\left(\frac{b}{c}\right)^k, \qquad  k\ge 1,
\]
where $\Gamma$ denotes the so-called Gamma function. The invariant distribution for the case $b>\rho$ is the same but instead of using the binomial theorem we need to use the generating function of a Beta-Geometric distribution with parameters $(1-\frac{\rho}{b}, \frac{\rho}{b})$.

In the case $\rho=b$, the representation of the function $f(z)=-\ln (1-z)$, for $z\in [0,1)$,  as an hypergeometric function  leads to 
\[
\nu(k)=\widetilde{C}_{c, b}\left(\frac{b}{c}\right)^k\frac{1}{k}, \qquad  k\ge 1,
\]
with 
\[
\widetilde{C}_{c, b}=\left(\ln \left(\frac{1}{1-\frac{b}{c}}\right)\right)^{-1}.
\]

The moment duality is also useful to show that a subcritical cooperative branching process with  interactions and no annihilation comes down from infinity.   Formally,  we define the law $\mathbf{P}_\infty$ starting from infinity with values in $\overline{\mathbb{N}}$ as the limits of the laws $\mathbf{P}_n$ of the process issued from $n$. When the limiting process is non-degenerate, it hits finite values in finite time with positive probability. Moreover, we say that the process $Z$ {\it comes down from infinity} if $\mathbf{P}_\infty(Z_t<\infty)=1$ for all $t>0$.\begin{theorem}\label{infinity}
Assume  $|\mathbf{m}|<\infty$ and  $a=0$. Then a subcritical cooperative (i.e. $\varsigma<0$)  branching process  with  interactions  $Z$ 
comes down from infinity. Furthermore,  we have that $$\mathbf{E}_\infty[\tau_1]<\infty,$$
where  $\tau_{1}=\inf\{t>0:Z_t=1\}$.
\end{theorem}
Another interesting consequence of  the moment duality and Proposition \ref{FLyap}, is the following result that describes the asymptotic behaviour of  jump-diffusions of the form (\ref{eq:dualnaintro}). In particular, we can determine the probability of fixation at 1 of the jump-diffusion $X$.

\begin{corollary}\label{thm:criteriaX}
Let $X$ be the unique strong solution of \eqref{eq:dualnaintro} which is the unique moment dual of a subcritical cooperative branching process with interactions $Z$ and no annihilation  (i.e. $\varsigma<0$ and $a=0$) and  $X_\infty:=\lim_{t\to\infty}X_t$. Then, for all $x\in[0,1]$, 
\begin{enumerate}
\item[i)] if $d=0$ and $\rho=0$,  $$\mathbb{P}_x(X_\infty=1)=x=1-\mathbb{P}_x(X_\infty=0),$$ 
\item[ii)] if $d=0$ and $\rho>0$,  $$\mathbb{P}_x(X_\infty=1)=\sum_{i\in \mathbb{N}} x^i \nu(i) ,$$ where $\nu$ is the unique stationary distribution of $Z$,
\item[iii)] if $d>0$,  $$\mathbb{P}_x(X_\infty=1)=1.$$ 
\end{enumerate}
\end{corollary}

 Finally, we study the case  $a>0$. This case has been studied before in different contexts, see for instance \cite{AS12,BEM07,BK11,B85} and the references therein. It is important to note that in this case monotonicity is lost, in the sense that it is not true that a bigger population has more probability of survival. For instance, if there are two individuals (or particles) the probability of extinction is higher than if there is only one individual (or one particle). Monotonicity is a very important and useful property, thus  the case $a>0$ seems to be technically more involved  and many properties of processes with annihilation events remain unknown.

We point out that   in this case there is also  a moment duality relationship between a branching process with interactions and  a jump-diffusion similar  to (\ref{eq:dualnaintro}). In this case,  the existence of a unique strong solution is more complicated than the case  $a=0$. We will not prove this claim, since in this case  there is no relevant interpretation for the moment dual of $Z$ and   we can study its long term behaviour   via a coupling argument. Moreover determining the invariant distribution of $Z$ (whenever it exist) in terms of the probability of fixation at, for example, $-1$ and $1$ seems more involved.

Finally, our last result  provide a  detailed description of the long term behaviour of $Z$ in the case when $a>0$. As it was claimed before, the behaviour of $Z$ is more involved than the case with no annihilation.

\begin{proposition}\label{lem:criteriaannihilation}
Let $Z$ be a subcritical cooperative branching process with  interactions (i.e. $\varsigma<0$) with $a>0$ and  let
\[
\tau_{0,1}=\inf\{t>0:Z_t\in\{0,1\}\} \qquad \textrm{and} \qquad \tau_{i}=\inf\{t>0:Z_t=i\},
\]
for $i\in\{0,1\}$.
\begin{enumerate}
\item[i)] If $d=0$, and 
\begin{enumerate}
\item[a)]   either $\Lambda(0,1]>0$,  $c>0$ or $\sum_{i\in \mathbb{N}}\pi_{2i-1}+\sum_{i\in \mathbb{N}} b_{2i-1}>0$,
\begin{enumerate}
\item[a.1)] with $\rho=0$,  then $Z$ gets absorbed in $\{0,1\}$ and $\sup_{n\in \mathbb{N}}\mathbf{E}_n[\tau_{0,1}]<\infty$,  
\item[a.2)] with $\rho>0$, then  $Z$ gets absorbed in  $\{0\}$ and   $\sup_{n\in \mathbb{N}}\mathbf{E}_n[\tau_0]<\infty$.
\end{enumerate}
 \item[b)]  $\Lambda(0,1]=0$, $c=0$ and $\sum_{i\in \mathbb{N}} \pi_{2i-1}+\sum_{i\in \mathbb{N}} b_{2i-1}=0$, 
\begin{enumerate}
\item[b.1)] with $Z_0=2n$ for some $n\in \mathbb{N}$, then  $Z$ takes values in the even numbers, gets absorbed in the state $\{0\}$ and $\sup_{n\in \mathbb{N}}\mathbf{E}_n[\tau_{0}]<\infty$,  
\item[b.2)] with $Z_0=2n+1$ for some $n\in \mathbb{N}_0$,  then  $Z$ takes values in the odd numbers. Further,
\begin{itemize}
\item[b.2.1)] if $\rho=0$, then the process $Z$ gets absorbed in the states $\{1\}$ and $\sup_{\in \mathbb{N}}\mathbf{E}_n[\tau_{1}]<\infty,$
\item[b.2.2)] if $\rho>0$, then the process $Z$ has unique stationary distribution $\nu$ satisfying
\[
\lim_{t\rightarrow \infty} \sup_{n\in \mathbb{N}}||\mathbf{P}_n(Z_t=\cdot)-\nu(\cdot)||=0.
\]
 
\end{itemize}
\end{enumerate}
\end{enumerate}
\item[ii)] If $d>0$, 
then $Z$ gets absorbed in the state $\{0\}$ and $\sup_{n\in \mathbb{N}}\mathbf{E}_n[\tau_{0}]<\infty$.
\end{enumerate}
\end{proposition}

 \begin{figure}[h]
\label{fig}
    \centering
     \includegraphics[height=.5\textwidth]{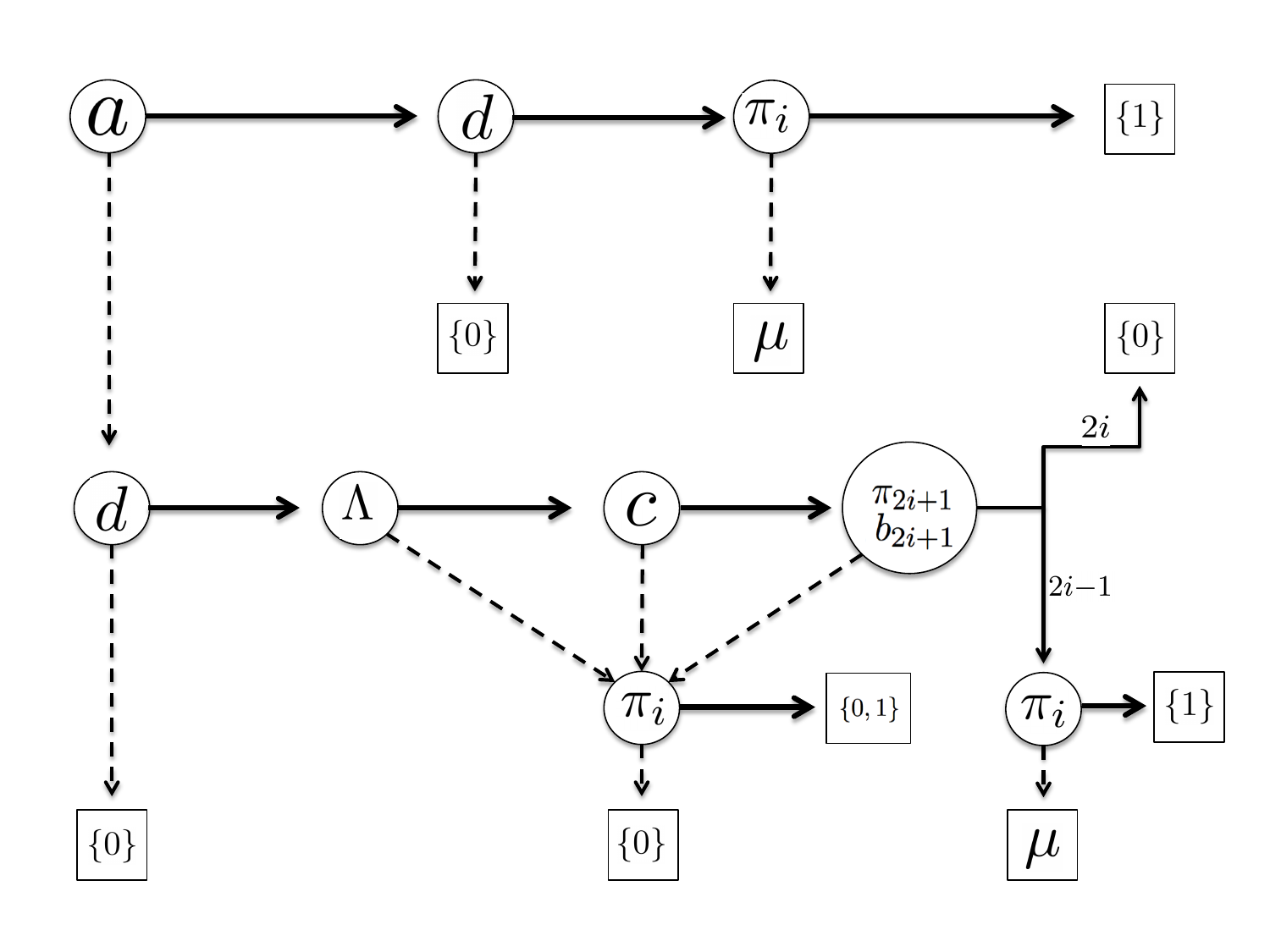}  
      \caption{\small This figure exemplifies the long term behaviour of  branching processes with interactions in the subcritical cooperative regime depending on its parameters. Doted lines mean that the parameter is different to zero while continuous lines mean that the parameter is zero. For example if $a=0$, $d=0$ and $\pi_i=0$ for all $i\in \mathbb{N}$, we take the right edge three times. In this case the process gets absorbed at $\{1\}$. The symbol $\mu$ means that the process has a non degenerated stationary distribution.}
\end{figure}

\section{Proofs}

\subsection{Proof of Theorem  \ref{Propfinitness}}
\begin{proof} 

For the first part of the statement, we apply Theorem 1.11 in  Chen \cite{Chen} (see also Theorem 2.2 in Tweedie \cite{Twee} for the case $\mathbf{m}\le 0$ and Theorem 2.1 in Meyn and Tweedie \cite{stability0} for the case $\mathbf{m}>0$). The proof of the aforementioned result follows from  a localisation argument and its conditions  can be written  in terms of the localized transition rates   unchanging the conclusion. More precisely, we consider the  particular case where the state space is $\mathbb{N}_0$ and introduce the sequence of sets $E_n=\{0, 1,2, \ldots, n\} $, for $n\in \mathbb{N}$,  which increases towards $\mathbb{N}_0$.  We also introduce the localized transition rates $(q^{(n)}_{i,j})_{i,j\in  \mathbb{N}_0}$ as follows
\[
q^{(n)}_{i,j}=\mathbf{1}_{E_n}(i)q_{i,j} \qquad\textrm{with}\qquad-q^{(n)}_{i,i}=\sum_{j\neq i}q^{(n)}_{i,j}\quad \textrm{for } i,j\in   \mathbb{N}_0.
\]
If a non-negative function $\varphi=(\varphi(i), i\in \mathbb{N}_0)$ satisfies
\[
\lim_{n\to\infty}\inf_{i\notin E_n}\varphi(i)=\infty,
\qquad
\textrm{and} 
\qquad
\sum_{j\in \mathbb{N}_0}q^{(n)}_{i,j}\Big(\varphi(j)-\varphi(i)\Big)\le c \varphi(i) \qquad \textrm{for}\,\, i\in E_n
\]
where  $c\in \mathbb{R}$, then the associated process to $(q_{i,j})_{i,j\in  \mathbb{N}_0}$ does not explode since the same arguments in   parts (ii), (iii) and (iv) of the proof  of  Theorem 1.11 in  Chen \cite{Chen}   still hold with the above assumptions.  In our particular case, we take  
\[
\varphi(i)=\left\{\begin{array}{ll}
i & \textrm{if } i\in E_n,\\
n & \textrm{otherwise.}
\end{array} \right .
\]
Since  $\inf_{i\notin E_n}\varphi(i)=n$ and $\varsigma \le 0$, we deduce 
\[
\sum_{j\in \mathbb{N}_0}q^{(n)}_{i,j}\Big(\varphi(j)-\varphi(i)\Big)\le i\mathbf{m} + i(i-1) \varsigma -\sum_{k=2}^ i (k-1)\binom{i}{k} \lambda_{i,k} \le \varphi(i) \mathbf{m}, \qquad \textrm{for}\,\, i\in E_n.
\]
 which implies from above that the process $Z$ does not explode (or it is unique or regular in the terminology of \cite{Chen}).

For the second statement, we assume that the measure $\Lambda$ is identically zero and that $\varsigma>0$. We first treat the case when  $d=\rho=0$. 
In order to do so, we  introduce  $A=(A_t, t\in \mathbb{R}_{+,0})$  a compound Poisson process with parameter  $a+b+c$ and jump distribution $\eta$ satisfying
\[
\eta(i)=\left\{ \begin{array}{ll}
\displaystyle\frac{a}{a+b+c} & \textrm{ if  } i=-2,\\[.25cm]
\displaystyle\frac{c}{a+b+c} & \textrm{ if  } i=-1,\\[.25cm]
\displaystyle\frac{b_i}{a+b+c} & \textrm{ if  } i\in \mathbb{N}.\\
\end{array}
\right .
\]
We  also consider the function $g(i)=i(i-1)$ which is  non negative for $i\in \mathbb{N}_{0}$ with  $g(i)=0$ for $i=0,1$. Hence, from Theorems 6.1.1 and 6.1.3  in Ethier and Kurtz \cite{EK} the random time changed process
\[
Z_{t}=A_{\int_0^t g(Z_s) \ud s},\qquad t\le \theta_{0,1} :=\inf\{s: A_s=0 \textrm{ or } 1\},
\]
defines a Markov process with  absorbing   states $\{0,1\}$ and whose infinitesimal generator is given by 
\[
Qf(n)=n(n-1)\mathcal{B}f(n), 
\]
where, for $f\in C_0(\mathbb{Z}\cup\{\infty\},  \mathbb{R})$ with $f(i)=0$ for $i\in \{\ldots, -2,-1, 0,1\}$, the infinitesimal generator $\mathcal{B}$ acts as follows
\[
\mathcal{B}f(n)=a[f(n-2)-f(n)] +   c[f(n-1)-f(n)]+\sum_{i\in \mathbb{N}} b_i[f(n+i)-f(n)], \quad n\in \mathbb{N}\setminus\{1\}.
\]
In other words, the process $Z$ has the same law as a branching process with interactions with parameters $a, c$ and  $b_i$, for $i\in \mathbb{N}$.

In the same probability space, we also consider a Poisson process $\widetilde{A}$ with parameter $\delta\in(0, \varsigma)$, then by using a similar random time change as above we introduce a branching process with interactions with parameters $c=a=0$, $b_1=\delta>0$ and $b_i=0$, for $i\in\mathbb{N}\setminus\{1\}$. We observe that  the process $\widetilde{Z}$, which is the process with this parameters in this probability space,  explodes  almost surely. Indeed, let us define 
 $\widetilde{\tau}^+_i=\inf\{t>0:\widetilde{Z}_t\geq i\}$ and observe that its life-time, here denoted by $\widetilde{\tau}_\infty$, satisfies $\widetilde{\tau}_\infty=\sup_{i\in \mathbb{N}}\widetilde{\tau}_i^+ $. Thus, we see
\begin{equation}
\widetilde{\mathbf{E}}_n[\widetilde{\tau}_\infty]\leq\widetilde{\mathbf{E}}_2[\widetilde{\tau}_\infty]=\sum_{i\in\mathbb{N}} \widetilde{\mathbf{E}}_i[\widetilde{\tau}^+_{i+1}-\widetilde{\tau}^+_{i}]=\frac{1}{\delta}\sum_{i\in\mathbb{N}\setminus\{1\}} \frac{1}{ i(i-1)}<\infty,
\end{equation}
where $\widetilde{\mathbf{E}}_n$ denotes the expectation of $\widetilde{Z}$ starting from $n\in\mathbb{N}\setminus\{1\}$.  The previous identity implies  that $\widetilde{\tau}_{\infty}$ is finite a.s or equivalently that $\widetilde{Z}$ explodes a.s.

  Next, we assume that $A_0=n>m=\widetilde{A}_0$ and define the event $E=\{A_t>\widetilde{A}_t, \text{ for all }t>0\}$. It is straightforward to deduce that the process $S_t=A_t-\widetilde{A}_t$, for $t\ge 0$,  is a compound Poisson process starting from $n-m>0$ and drifting to $\infty$, since $\varsigma>\delta$. Hence with positive probability the process $S$ stays strictly positive, implying that $\mathbb{P}(E|S_0=n-m)>0$. In other words, conditionally on $E$ and given that $S_0=n-m>0$, we have that the process $\widetilde{Z}$ is stochastically dominated by $Z$ and since the former explodes then $Z$ also explodes. This proves our statement for the case $d=\rho=0$.

For the case $d+\rho>0$, our analysis  will be based in the following coupling argument. Let us consider two branching processes with  interactions, $Z=(Z_t, t\in \mathbb{R}_{+,0})$ and $\overline{Z}=(\overline{Z}_t, t\in \mathbb{R}_{+, 0})$, with respective parameters $d, a, c,b_i$ and $\pi_i$, for $i\in \mathbb{N}$,  such that  $d+\rho>0$; and $\overline{d}=\sum_{i\in \mathbb{N}}\overline{\pi_i}=0$,
 $\overline{c}=c+\epsilon d,$ $\overline{a}=a$ and  $\overline{b_i}=b_i$, for all $i\in \mathbb{N}$, where $\epsilon$ is chosen positive and such that
\begin{equation}\label{eqepsilon}
-2\overline{a}-\overline{c}+\sum_{i\in \mathbb{N}} i \overline{b_i} =-2a-c-\epsilon d+\sum_{i\in \mathbb{N}} i b_i> 0.
\end{equation}
We denote by $\overline{\mathbf{P}}_n$  the  law of $\overline{Z}$ starting from $\overline{Z}_0=n$. We also observe, from above,  that $\overline{Z}$ explodes with positive probability. Moreover, let $n_0=\inf\{n\in \mathbb{N}: 1< (n-1)\epsilon\}$ and note  that for any starting condition $Z_0=\overline{Z}_0=n> n_0$, we necessarily have 
 \[
 nd+n(n-1)c<n(n-1)\overline{c}=\epsilon n(n-1)d+n(n-1)c.
 \] 
Next,  we  consider the following coupling between the Markov chains  $Z$ and $\overline{Z}$. Let us introduce  the process $\{(U_t,\overline{U}_t), t\in \mathbb{R}_{+,0}\}$  with the following dynamics: Assume that $U_0=\overline{U}_0>n_0$, for all $t<\gamma_{n_0}:=\inf\{t>0: \overline{U}_t\le n_0\}$, the couple $(U_t,\overline{U}_t)$ has the following transitions rates,  for $m\ge n$, 
\begin{equation}\label{coupling}
(U_t,\overline{U}_t)\text{ goes from }(n,m)\text{ to}\left\{ \begin{array}{ll}
(n+i,m+i), & \textrm{ with rate } m(m-1)b_i,\\
(n+i,m), & \textrm{ with rate } n\pi_i+[n(n-1)-m(m-1)]b_i,\\
(n-1,m-1), & \textrm{ with rate } \underline{\Theta}(n,m),\\
(n,m-1), & \textrm{ with rate }  m(m-1)(\epsilon d+c)-\underline{\Theta}(n,m),\\
(n-1,m), & \textrm{ with rate }  n(n-1)c+nd-\underline{\Theta}(n,m),\\
(n-2,m-2)&\textrm{ with rate } n(n-1)a.
\end{array}
\right .
\end{equation}
where $\underline{\Theta}(n,m)=\min( n(n-1)c+nd,m(m-1)(\epsilon d+c))$.
For  $t\geq \gamma_{n_0}$, the Markov chains $U_t$ and $\overline{U}_t$ evolve independently with the same transition rates as $Z$ and $\overline{Z}$, respectively. We denote by $\mathbf{P}_{(n,m)}$ the law of $\{(U_t,\overline{U}_t), t\in \mathbb{R}_{+,0}\}$ starting from $(n,m)$. Note  from  our construction, that the Markov chains $U$ and $\overline{U}$ are equal in distribution to $Z$ and  $\overline{Z}$, respectively. Furthermore, we necessarily  have
$$
\mathbf{P}_{(n,n)}\Big(U_t\geq \overline{U_t}, \textrm{ for } t<\gamma_{n_0}\Big)=1.$$
We deduce for $n>n_0$, $\mathbf{P}_{(n, n)}(U_t=\infty)>\mathbf{P}_{(n, n)}(\overline{U}_t=\infty,\gamma_{n_0}>t)$ where  the right hand side of the inequality is positive for $t$ sufficiently large. The latter implies that $Z$ is not conservative.

\end{proof}

For the proof of Proposition \ref{FLyap}, we verify the Foster-Lyapunov conditions for positive recurrence. In order to do so, we consider the jump Markov chain  associate to $Z$ that we denote by $Y=(Y_n, n\in \mathbb{N}_0)$ and whose jump matrix is given by $P=(p_{i,j})_{i,j\in \mathbb{N}_0}$ where
\[
p_{i,j}=\left\{ \begin{array}{ll}
    -q_{i,j}/q_{i,i} & \textrm{if $j\ne i$}\\
0 & \textrm{if $j=i$,}
\end{array} \right.
\]
for $i\in \mathbb{N}$ and $p_{0,0}=1$.  On the other hand, we recall that if a state is recurrent or transient for the jump chain $Y$ then the state is also recurrent or transient for the chain $Z$. See for instance the monograph of Norris \cite{No} for further details of Markov chains and its associated jump chain.

\begin{proof}[Proof of Proposition \ref{FLyap}] We first prove part (ii). Recall that $d=0$ and $\rho>0$ and  that under our assumption $c>0$. From the form of the jump matrix $P$ defined above, we observe that the jump chain $Y$ is irreducible in $\mathbb{N}$ and that the state $\{0\}$ is not accesible from any state  $i\in\mathbb{N}$.

In order to analyse the class properties of the state space of the jump chain $Y$, and implicitly for $Z$,  we  apply the Foster-Lyapunov criteria (see for instance Proposition 1.3 in Hairer \cite{Ha} or Theorem 4.2 in Meyn and Tweedie \cite{stability0}). Let us consider the discrete generator associated to $Y$ which is defined as follows
\[
\mathcal{L}f(i)= Pf(i)-f(i)=-\frac{1}{q_{i,i}}\sum_{j\in \mathbb{N}_0}q_{i,j}f(j),
\]
for $f:\mathbb{N}_0\to \mathbb{R}$. We choose the Lyapunov function to be  $f(i)=C i$ for $i\in \mathbb{N}_0$ and $C$ a positive constant that we will specify later. According to the  Foster-Lyapunov criteria we know that  if $\mathcal{L}f(i)\leq -1$ for all but finitely many values of $i$, then $Y$ is positive recurrent.  Hence straightforward computations leads to 
\begin{equation}\label{id}
\begin{split}
\mathcal{L}f(i)&=-\frac{C}{q_{i,i}}\left(i\mathbf{m} + i(i-1) \varsigma-\sum_{k=2}^ i (k-1)\binom{i}{k} \lambda_{i,k} \right)\\
&\le C\left( \frac{m}{(i-1)(b+c+a)}+\frac{\varsigma}{(b+c+a)}-\frac{\sum_{k=2}^ i (k-1)\binom{i}{k} \lambda_{i,k}}{i(i-1)(b+c+a)}\right).
\end{split}
\end{equation}
For $i$ sufficiently large,  we can clearly choose $C$ in such a way that $\mathcal{L}f(i)\le -1$, implying that  the Foster-Lyapunov criteria holds. In other words,  $Z$ is positive recurrent in $\mathbb{N}$. 

Now we prove part (i). Recall that $\rho=0$ and $d=0$  and  that under our assumption $c>0$. Again from the form of the jump matrix $P$, we observe that the state  $\{1\}$ is absorbing and that $\{0\}$ is not accessible.  Moreover, if $b>0$  the class $\{2, 3, \ldots\}$ is irreducible for $Z$, and  if $b=0$, the process $Z$ is decreasing. In any case, we can always stochastically dominate $Z$ with another branching process with interactions, here denoted by $\widetilde{Z}$, with the same parameters except  $\pi_1>0$.  We refer to Section 2 of L\'opez et al. \cite{LMS} for the conditions on the intensities $(q_{i,j})_{i,j\ge 0}$ in order to have  existence of an order preserving coupling. Since $\widetilde{Z}$ is recurrent and reach the state $\{1\}$ in finite time a.s., then $Z$ is absorbed at $1$ in finite time a.s.

Finally, we deduce part (iii). Since in this case $d>0$, from the form of the jump matrix $P$ we observe that the state  $\{0\}$ is accesible.  If $\rho=0$ and $b=0$, the process $Z$ is decreasing and  if $\rho>0$ or $b>0$, then $\mathbb{N}$ is irreducible for $\bar Z$ constucted in part (ii) and again an order preserving coupling argument guarantees that $Z$ is absorbed at $\{0\}$ in finite time a.s.

\end{proof}

\subsection{Moment duality property and proof of Theorem \ref{lemma:characterization}.}
The aim of this section is to prove Theorem  \ref{lemma:characterization} which characterizes a class of jump-diffusions  which fulfills the moment duality property with respect to (sub)critical branching processes with interactions and without annihilation, i.e. $a=0$. This moment duality will be very helpful for our purposes.  The proof of Theorem  \ref{lemma:characterization} follows Lemmas \ref{lemma3} and \ref{lemma1}  below.

In order to prove Theorem \ref{lemma:characterization} we first introduce the class of jump-diffusions via the following SDE
\begin{equation}\label{eq:dualna}
\ud X_t=\mu(X_{t})\ud t+\sigma(X_t)\ud B_t +\int_{(0,1]}\int_{(0,1]}zg(X_{t-}, u) \widetilde{N}(\ud t,\ud z,\ud u),
\end{equation}
where $\widetilde{N}$ is a compensated  Poisson random measure on  $(0,\infty)\times (0,1]^2$ with intensity $\ud t z^{-2}\Lambda(\ud z) \ud u$, the function $g(x,u)$ is defined by
\[
g(x,u)=\mathbf{1}_{\{u\le 1\land x \}}-(1\land x)\mathbf{1}_{\{x\ge 0\}},
\]
the functions $\mu:\R\rightarrow \R$ and $\sigma: \R\rightarrow \R$ are continuous and  satisfy
\[
\mu(x)=d\mathbf{1}_{\{x< 0\}}+\left(d(1-x)+\sum_{i\in \mathbb{N}}\pi_{i}(x^{i+1}-x)\right)\mathbf{1}_{\{x\in[0,1]\}},
\]
and
\[
\sigma^2(x)=\left(c(x-x^2)+\sum_{i\in \mathbb{N}} b_i(x^{i+2}-x^2)\right)\mathbf{1}_{\{x\in[0,1]\}},
\]
for some positive constants $d,c,b_i, \pi_i$, for $i\in \mathbb{N}$.

 \begin{lemma}\label{lemma3}
 Fix $d,c, b_i,\pi_i$, for $i\in \mathbb{N}$,  where all terms  are no negative. If such coefficients satisfy $|\mathbf{m}|<\infty$ and $\varsigma\le 0$,
then the SDE (\ref{eq:dualna})  with starting point  $X_0\in[0,1]$ has a unique strong solution taking values on $[0,1]$.  \end{lemma}
 \begin{proof} We first observe that if $X$ satisfies (\ref{eq:dualna}) with $X_0\in [0,1]$, a.s., then $X_t\in[0,1]$ for all $t\in \mathbb{R}_+$, a.s.  
  Suppose that there exist $\epsilon>0$ such that $\tau:= \inf\{t\ge 0: X_t\le -\epsilon\}$ is finite with positive probability. Then on the event $\{\tau<\infty\}$, we have $X_{\tau}=-\epsilon$
and 
\[
\tau>\tau_0=\inf\{s<\tau: X_t\le 0, \forall s\le t\le \tau \}.
\] So we take $r\ge 0$ such that $\{\tau_0\le r<\tau\}$ occurs with strictly positive probability and on that event, we observe
\[
X_{t\land \tau}=X_{r\land \tau}+ d(t\land \tau -r\land \tau), \qquad t\ge r,
\]
 implying that $t\mapsto X_{t\land \tau}$ is non-decreasing on $[r,\infty)$. Since $X_r>-\epsilon$ on $\{\tau_0\le r<\tau\}$, we get a contradiction. The previous argument implies that $X_t\ge 0$, a.s. A similar argument proves that $X_t\le 1$, a.s. We leave the details to the reader.
  
 Now, we show that $X$ has a unique strong solution on $[0,1]$. In order to do so, we first observe that for $x, y \in [0,1]$
  \[
  |\mu(x)-\mu(y)|\le \left(d + \sum_{i\in \mathbb{N}}(i+1)\pi_i\right)|x-y|.
 \]
The latter inequality implies that  $\mu$ is Lipschitz continuous on $[0,1]$.  

 On the other hand, we also observe that  $\sigma$ satisfies, for all $x,y\in [0,1]$ ,
\begin{align*}
|\sigma^2(x)-\sigma^2(y)|&\leq \left|c(x-y)+c(y^2-x^2)+\sum_{i\in \mathbb{N}} b_i\Big(x^{2+i}-y^{i+2} -(x^{2}-y^{2})\Big)\right|\\
&\leq \left(3c+\sum_{i\in \mathbb{N}} b_i(i+2)+2\sum_{i\in \mathbb{N}} b_i\right)|x-y|,
\end{align*}
 implying from our assumptions  that $\sigma$ is H\"older continuous on $[0,1]$ with exponent 1/2.
 
 Finally, we observe from Corollary 6.2  in Li and Pu \cite{LiPu}, that $x\mapsto x+g(x, u)$ is non-decreasing and for any $0\le x, y\le 1$,
 \[
 \int_{(0,1]}\frac{\Lambda(\ud z)}{z^2}\int_0^1 \ud r\, z^2|g(x,r)-g(y,r)|\le |x-y|\int_{(0,1]}\Lambda(\ud z).
 \] 
Hence conditions (3.a),(3.b) and (5.a) in Li and Pu \cite{LiPu} are satisfied, implying that there is a unique strong solution to  the SDE (\ref{eq:dualna}).
 \end{proof}

 Let $X=(X_t, t\in \mathbb{R}_{+,0})$  be the unique strong solution of the SDE (\ref{eq:dualna}). We denote by $\mathcal{A}$ for its infinitesimal generator of the process X, which satisifes
\[
\mathcal{A}f(x):=\lim_{t\downarrow 0}\frac{\mathbb{E}_{x}[f(X_t)]-f(x)}{t}
\]
for any $f:[0,1]\mapsto\mathbb{R}$  continuous such that the limit exist uniformly on $x\in[0,1]$. If such limit exist then we say that $f\in\mathcal{D}_{\mathcal{A}}$, the domain of the  infinitesimal generator $\mathcal{A}$. 
Observe, from It\^o's formula,  that the  infinitesimal generator $\mathcal{A}$ satisfies  for $f\in C^2_b([0,1], \mathbb{R})$, the set of twice continuously differentiable bounded functions, and $x\in [0,1]$
\[
\begin{split}
\mathcal{A}f(x)&=\mu(x)f^\prime(x)+\frac{\sigma^2(x)}{2}f^{\prime\prime}(x)\\
&\hspace{1.5cm}+\int_{[0,1]}\Big(xf(x(1-z)+z)+(1-x)f(x(1-z))-f(x)\Big)\frac{\Lambda(\ud z)}{z^2},
\end{split}
\]
where $\mu, \sigma:[0,1]\to\mathbb{R}$  and $\Lambda$ are defined as above. In other words, for $f\in C^2_b([0,1], \mathbb{R})$, the process
\[
f(X_t)-\int_0^t\mathcal{A}f(X_s)\ud s, 
\]
is a martingale. On the other hand, we also observe that $C^2_b([0,1], \mathbb{R})$ is dense in $C([0,1], \mathbb{R})$, the set of continuous functions from $[0,1]$ to $\mathbb{R}$. It is not so difficult to see that the generator $\mathcal{A}$ satisfies the positive maximum principle. Since $[0,1]$ is  a compact space  and the space of all polynomials on [0,1] is dense on $C([0,1], \mathbb{R})$, the Hille-Yosida theorem implies that the closure of $\mathcal{A}$ restricted on $C^2_b([0,1], \mathbb{R})$ generates a strongly continuous, time homogenous, positive contraction Feller-semigroup on $C([0,1], \mathbb{R})$, see for instance Theorem 4.2.2 and Lemma 4.2.3 in \cite{EK}. Theorem 4.1.1 and Proposition 4.2.9 imply that the process $X$ is Feller.

 Recall that  $\mathbb{P}_x$  and $\mathbf{P}_n,$ denote the laws of $X$ starting from $x\in[0,1]$ and   $Z$ starting from $n\ge 0$, respectively. Moreover, from our assumptions we have  that $\mathbf{P}_n(Z_t<\infty)=1$ for any $n\in\mathbb{N}_0$ and $t\in\mathbb{R}_{+,0}$.

In practice, an effective methodology to show duality between two Markov processes is via their infinitesimal generator.  Indeed, the following result  is  a  special case of Theorem 4.11 and Corollary 4.15 in Ethier and  Kurtz \cite{EK}  and can also be seen as a small modification of Proposition 1.2 of Jansen and Kurt  in \cite{JK14}.

\begin{proposition}\label{prop:provingduality}
Let $(Y^{(1)}_t, t\in\mathbb{R}_{+,0})$ and $(Y^{(2)}_t, t\in\mathbb{R}_{+,0})$ be two Markov processes  taking values on $E_1$ and $ E_2$, respectively. Let $H:E_1\times E_2\rightarrow \mathbb{R}$ be a bounded and continuous function and assume that there exist functions $g_i:E_1\times E_2\mapsto \mathbb{R}$, for $i=1,2$, such that for every $n\in E_1,x\in E_2$ and every $T>0$,  the processes $(M_t^{(1)}, 0\le t\le T)$ and $(M_t^{(2)}, 0\le t\le T)$, defined as follows
\begin{eqnarray}
M_t^{(1)}:=H(Y^{(1)}_t,x)&-&\int_0^t g_1(Y^{(1)}_s,x)\ud s\label{MG1}\\
M_t^{(2)}:=H(n,Y^{(2)}_t)&-&\int_0^t g_2(n, Y^{(2)}_s)\ud s \label{MG2}
\end{eqnarray}
are martingales with respect to  the natural filtration of $Y^{(1)}_t$ and $Y^{(2)}_t$ respetively. Then, if 
$$
g_1(n,x)=g_2(n,x)\qquad \textrm{for all}\qquad  n\in E_1,x\in E_2,
$$
the processes $(Y^{(1)}_t, t\in\mathbb{R}_{+,0})$ and $(Y^{(2)}, t\in\mathbb{R}_{+,0})$ are dual with respect to $H$. 
\end{proposition}
This proposition is a direct consequence of Theorem 4.11 in Ethier Kurtz \cite{EK} taking $H$ bounded and continuous and $\alpha=\beta=0$.

If the function $H$ is such that $H(n,x)=x^n$, for $x\in [0,1]$ and $n\in \mathbb{N}_0$,   the above duality relationship is known as {\it moment duality.} The moment duality between $X$ and $Z$ reads as follows.
\begin{lemma}\label{lemma1} Assume $|\mathbf{m}|<\infty$, $\varsigma\le 0$ and $a=0$. Let   $X$ be the unique strong solution of the SDE (\ref{eq:dualna}) then it  is the unique (in distribution) moment dual of $Z$ a (sub)critical branching process with  interactions  and no annihilation. More precisely, for $x\in [0,1]$ and $n\in\mathbb{N}_{0}$, we have 

$$
\ensuremath{\mathbb{E}}_x[X_t^n]=\mathbf{E}_n[x^{Z_t}]\qquad\text{for}\quad  t\in\mathbb{R}_{+,0}. 
$$

\end{lemma}
\begin{proof}  We take $H(n,x)=x^n$, for $x\in [0,1]$ and $n\in \mathbb{N}$, and  when $n=0$, $H(0,x):=1$ for $x\in[0,1]$, which are continuous and bounded.

Now, we  claim that for  fixed $x\in[0,1]$, $H(n,x)$ is in the domain of $Q$ and for fixed $n\in \mathbb{N}_0$, it is in the domain of $\mathcal{A}$. The latter follows from the fact that $H(n,\cdot)$ is a polynomial on $[0,1]$ which clearly belongs to $C^2_b([0,1], \mathbb{R})$. Implicitly, we have
\[
H(n,X_t)-\int_0^t\mathcal{A}H(n, X_s)\ud s, 
\]
is a martingale.  The first claim follows  by recalling that linear combinations of exponential functions as in \eqref{lcexp} belongs to the domain of $Q$ (see Section \ref{intro}) and the fact that $H(n,x)=e^{n\log x }$, imply that the process
\[
H(Z_t,x)-\int_0^t QH(Z_s, x)\ud s, 
\]
is also a martingale.

Next, we observe that  for $x\in(0,1)$ and $n\in \mathbb{N}\setminus\{1\}$, 
\begin{equation*}
\begin{split}
\mathcal{A}H(n,x)&=nd(x^{n-1}-x^n)+n\sum_{i\in \mathbb{N}}\pi_{i}(x^{n+i}-x^n)\\
&\hspace{1cm}+n(n-1)\sum_{i\in \mathbb{N}} b_i(x^{n+i}-x^n)+n(n-1)c(x^{n-1}-x^n)\\
&\hspace{1.8cm}+\int_{[0,1]}\Big(x(x(1-z)+z)^n+(1-x)(x(1-z))^n-x^n\Big)\frac{\Lambda(\ud z)}{z^2}.
\end{split}
\end{equation*}
On the other hand,  let  $B$ be a Bernoulli r.v. taking the value 1 with probability $x$ and 0 with probability $1-x$. Thus, for $n\in \mathbb{N}\setminus\{1\}$, we observe 
\[
\begin{split}
x(x(1-z)+z)^n&+(1-x)(x(1-z))^n-x^n=\mathbb{E}\Big[(x(1-z)+Bz)^n\Big]-x^n\\
&=\mathbb{E}\left[\sum_{k=0}^n \dbinom{n}{k}x^{n-k}(1-z)^{n-k}z^k B^k\right]-x^n\\
&=\sum_{k=2}^n \dbinom{n}{k}x^{n-k+1}(1-z)^{n-k}z^k +x^n(1-z)^n+nx^n(1-z)^{n-1}z-x^n\\
&=\sum_{k=2}^n \dbinom{n}{k}x^{n-k+1}(1-z)^{n-k}z^k -x^n\Big(1-(1-z)^n-n(1-z)^{n-1}z\Big)\\
&=\sum_{k=2}^n \dbinom{n}{k}(1-z)^{n-k}z^k\Big(x^{n-k+1} -x^n\Big).\\
\end{split}
\]
Hence,   for $n\in\mathbb{N}\setminus\{1\}$, we deduce
\[
\int_{[0,1]}\Big(x(x(1-z)+z)^n+(1-x)(x(1-z))^n-x^n\Big)\frac{\Lambda(\ud z)}{z^2}=\sum_{k=2}^n\binom{n}{k}\lambda_{n,k}[x^{n-k+1}-x^n],
\]
implying that $\mathcal{A}H(n, x)=QH(n, x)$ for $x\in(0,1)$ and $n\in \mathbb{N}\setminus\{1\}$.   For the case $n=1$ and $x\in(0,1)$, we have 
\[
\mathcal{A}H(1, x)=d(1-x)+\sum_{i\in \mathbb{N}}\pi_{i}(x^{1+i}-x)=QH(1, x).
\]
Finally, we consider the remaining cases. For  $n\in\mathbb{N}$ and  $x=0$ or $x=1$,  it is clear  $\mathcal{A}H(n, x)=0=QH(n, x)$. Moreover, when $n=0$, we  have $QH(0, x)=0=\mathcal{A}H(0, x)$ for $x\in[0,1]$. 

Putting all pieces together, we deduce that  $\mathcal{A}H(n, x)=QH(n, x)$ for $x\in[0,1]$ and $n\in \mathbb{N}_0$.  Hence we can apply  Proposition \ref{prop:provingduality} and deduce the moment duality between $X$ and $Z$.
\end{proof}

We finish this section with the following useful Lemma which provides an upper bound for the expected fixation time of the process $X$ defined by (\ref{eq:dualna}) with $\mu\equiv0$ and $\Lambda\equiv 0$. We will use this Lemma for the proof of Theorems \ref{thm:criteria} and \ref{infinity}. In particular, we have that the fixation time  for  $X$ with $\mu\equiv 0$, $\Lambda\equiv 0$ and starting at $x\in[0,1]$ is finite $\mathbb{P}_x$-a.s.
\begin{lemma}\label{lemma:hittingtime1}
Let  $X$ be the process defined by the SDE (\ref{eq:dualna})  with $\varsigma<0$, $\Lambda\equiv 0$, $d=0$, $\pi_i=0$, for all $i\in \mathbb{N}$ and recall  $T_{0,1}=\inf\{t>0:X_t\in\{0,1\}\}$. Then  for any $x\in[0,1]$ 
\[
\ensuremath{\mathbb{E}}_x[T_{0,1}]\le -\frac{2}{c-\sum_{i=1}^\infty i b_i }\Big((1-x)\ln (1-x) +x\ln x\Big).
\]
Moreover, we have  $\mathbb{P}_x(X_{T_{0,1}}=1)=x$.
\end{lemma}
\begin{proof} Since $\Lambda\equiv 0$, $d=0$ and $\pi_i=0$, for all $i\in \mathbb{N}$, we observe that the process $X$, under $\mathbb{P}_x$, satisfies the following SDE
\begin{equation}\label{diffmart}
X_t=x+\int_0^t\sigma(X_s)\ud B_s,
\end{equation}
where $\sigma$ is defined as above. The process $X$ is a diffusion taking  values on $[0,1]$ with 
 $\{0,1\}$ as absorbing states. Observe that its associated scale function is given by $S(y)=y+\mathbf{c}$, where $\mathbf{c}$ is a constant, and its associated speed measure satisfies
 \[
 m(\ud y)=\frac{2}{\sigma^2(y)}\ud y.
 \]
 Hence from Corollary VII 3.8 in \cite{RevuzYor}, we have 
 \[
 \mathbb{E}_x[T_{0,1}]=2\int_0^1 G(x,y)\frac{1}{\sigma^2(y)}\ud y, 
 \]
 where
 \[
 G(x,y)=\left\{ \begin{array}{ll}
 x(1-y) & \textrm{ if } 0\le x\le y\le 1,\\
 y(1-x) & \textrm{ if } 0\le y\le x\le 1,\\
 0 & \textrm{ otherwise.}
 \end{array}
 \right .
 \]
 Now, we observe that for $y\in[0,1]$, the following inequality holds
\[
\sigma^2(y)\ge \Big( c-\sum_{i\in \mathbb{N}} i b_i \Big)y (1-y),
\]
implying 
\[
\begin{split}
\mathbb{E}_x[T_{0,1}]&\le \frac{2}{\Big( c-\sum_{i\in \mathbb{N}}^\infty i b_i \Big)}\left(\int_0^x \frac{y(1-x)}{y(1-y)}\ud y +\int_x^1 \frac{x(1-y)}{y(1-y)}\ud y\right)\\
&=- \frac{2}{\Big( c-\sum_{i\in \mathbb{N}} i b_i \Big)}\Big((1-x)\ln (1-x) +x\ln x\Big),
\end{split}
\]
which provides the desired inequality.

Finally, since $X$ is a local martingale taking values on $[0,1]$, it  is a uniformly integrable martingale. Hence from  the optional stopping Theorem, we deduce
\[
x=\mathbb{E}_1[X_0]=\mathbb{E}_x[X_{ T_{0,1}}]=\mathbb{P}_x(X_{T_{0,1}}=1).
\]
Now the proof is completed.
\end{proof}

\subsection{Proof of Theorem \ref{thm:criteria} \& Corollary \ref{thm:criteriaX} }
For the proof of Theorem \ref{thm:criteria}, we need the following Lemma.
\begin{lemma}\label{lem:criteria}
Assume that there is no annihilation, no death and no branching, i.e. $\varsigma<0$, $a=0$, $d=0$ and $\rho=0$. Then
\[
\sup_{n\in \mathbb{N}}\mathbf{E}_n\big[\tau_1\big]<\infty,
\]
where $\tau_{1}=\inf\{t>0:Z_t=1\}$.
\end{lemma}
\begin{proof}
In order to prove  the statement,  it is enough to show the case when there are no catastrophes since  an order preserving coupling can be established. Indeed, we can stochastically dominate the case with catastrophes with another branching processes with interactions with the same parameters but  with  no catastrophes. We refer to Section 2 of L\'opez et al. \cite{LMS} for the conditions on the intensities $(q_{i,j})_{i,j\in\mathbb{N}_0}$ in order to have  existence of an order preserving coupling.

From the  moment duality, i.e. Lemma \ref{lemma:characterization}, we get
\[
\mathbf{E}_n[x^{Z_t}]=\mathbb{E}_x[X^n_{ t}]=\mathbb{E}_x\Big[X^n_{ t}\mathbf{1}_{\{T_{0,1}\ge t\}}\Big]+\mathbb{E}_x\Big[X^n_{ T_{0,1}}\mathbf{1}_{\{T_{0,1}<t\}}\Big].
\]
Thus by taking $t$ goes to $\infty$, the dominated convergence Theorem and recalling that $T_{0,1}$ is finite $\mathbb{P}_x$-a.s. (see Lemma \ref{lemma:hittingtime1}), we observe
\begin{equation}\label{limitx}
\lim_{t\to \infty}\mathbf{E}_n[x^{Z_t}]=\mathbb{P}_x(X_{T_{0,1}}=1)=x.
\end{equation}
On the other hand, again from the moment duality between $X$ and $Z$, we deduce the following inequality
\[
\mathbf{P}_n(Z_t=1)=\frac{\mathbf{E}_n[x^{Z_t}]-\sum_{i=2}^\infty x^{i}\mathbf{P}_n(Z_t=i)}{x}\ge \frac{\mathbf{E}_n[x^{Z_t}]}{x}-x.
\]
Therefore from identity (\ref{limitx}), we get
$$
\lim_{t\rightarrow \infty}\mathbf{P}_n(Z_t=1)\geq 1-x.
$$
Implicitly,  there are $t_1>0$ and $\epsilon>0$ such that for every $n\in\mathbb{N},$
$$
\mathbf{P}_n(Z_{t_1}=1)>\epsilon.
$$
Next we observe, using the Markov property of $Z$ starting from $n\in \mathbb{N}$ at times $(j t_1)_{j\in \mathbb{N}}$, that $\sigma_1$ is stochastically dominated by   a Geometric random variable  with parameter $\epsilon$, here denoted by $\Upsilon_\epsilon$,  that bounds the number of steps that takes to the discretized Markov chain $(Z_{jt_1}, j\ge 0)$ to reach the state 1. In other words, for any $n,m\in \mathbb{N}$, we have
\[
\mathbf{P}_n(\tau_1> mt_1)\le \mathbb{P}(\Upsilon_\epsilon > mt_1).
\]
The latter inequality implies, in particular, that
$$
\sup_{n\in \mathbb{N}}\mathbf{E}_n\big[\tau_1\big]<\frac{t_1}{\epsilon},
$$
which completes the proof.
\end{proof}

\begin{proof}[Proof of Theorem  \ref{thm:criteria} ]
We first deal with the case with no catastrophes i.e. $\Lambda\equiv 0$. From our assumptions $Z$ has parameters $d=0,$ $c,\pi_i,b_i$,  for $i\in \mathbb{N}$,   such that  $\rho=\sum_{i\in \mathbb{N}} \pi_i>0$ and $\varsigma<0$. Our arguments will be based on a coupling argument.  More precisely, let us consider two branching processes with  interactions, $Z=(Z_t, t\in \mathbb{R}_{+,0})$ and $\overline{Z}=(\overline{Z}_t, t\in \mathbb{R}_{+,0})$, with respective parameters $a=d=0, c,b_i$ and $\pi_i$, for $i\in \mathbb{N}$,  such that  $\rho>0$; and $\overline{d}=\overline{a}=\sum_{i\in \mathbb{N}}\overline{\pi_i}=0$,
 $\overline{c}=c$ and  $\overline{b_i}=b_i+\epsilon \pi_i$, for all $i\in \mathbb{N}$, where $\epsilon$ is chosen positive and such that
\begin{equation}\label{eqepsilon}
-\overline{c}+\sum_{i\in \mathbb{N}} i \overline{b_i} =-c+\sum_{i\in \mathbb{N}}  i b_i +\epsilon \sum_{i\in \mathbb{N}}  i \pi_i< 0.
\end{equation}
We denote by $\overline{\mathbf{P}}_n$  the  law of $\overline{Z}$ starting from $\overline{Z}_0=n$ and observe, from the first part of the proof and under our assumptions,  that $\overline{Z}$ 
does not explode and gets absorbed in state $\{1\}$. Moreover, let $n_0=\inf\{n\in \mathbb{N}: 2< (n-1)\epsilon\}$ and observe  that for any starting condition $Z_0=\overline{Z}_0=n> n_0$, we necessarily have 
 \[
 n\pi_i+n(n-1)b_i<n(n-1)\overline{b}_i=\epsilon n(n-1)\pi_i+n(n-1)b_i.
 \]

Next,  we  consider the following coupling between the Markov chains  $Z$ and $\overline{Z}$.  Let us introduce  the process $\{(U_t,\overline{U}_t), t\in \mathbb{R}_{+,0}\}$  with the following dynamics: for all $t<\gamma_{n_0}:=\inf\{t>0: U_t\le n_0\}$, the couple $(U_t,\overline{U}_t)$ has the following transitions rates,  for $m\ge n$
\begin{equation}\label{coupling}
(U_t,\overline{U}_t)\text{ goes from }(n,m)\text{ to}\left\{ \begin{array}{ll}
(n+i,m+i), & \textrm{ with rate } n\pi_i+n(n-1)b_i,\\
(n,m+i), & \textrm{ with rate } m(m-1)\overline{b}_i-n\pi_i+n(n-1)b_i,\\
(n-1,m-1), & \textrm{ with rate } n(n-1)c,\\
(n,m-1), & \textrm{ with rate } m(m-1)c-n(n-1)c,\\

\end{array}
\right .
\end{equation}
and for  $t\geq \gamma_{n_0}$, the Markov chains $U_t$ and $\overline{U}_t$ evolve independently with the same transition rates as $Z$ and $\overline{Z}$, respectively. We denote by $\mathbf{P}_{(n,m)}$ the law of $\{(U_t,\overline{U}_t), t\in \mathbb{R}_{+,0}\}$ starting from $(n,m)$. Note that from  our construction, the Markov chains $U$ and $\overline{U}$ are equal in distribution to $Z$ and  $\overline{Z}$, respectively.  We refer to Chapters IV and V in Lindvall \cite{Lind} for further details of this type of couplings and stochastic dominance.

 Furthermore, we necessarily  have for all $n>n_0$
$$
\mathbf{P}_{(n,n)}\Big(U_t\leq \overline{U_t}, \textrm{ for } t<\gamma_{n_0}\Big)=1,$$
where $\mathbf{P}_{(n,n)}$ denotes the law of $\{(U_t,\overline{U}_t), t\in \mathbb{R}_{+,0}\}$ starting from $(n, n)$. Hence, from  Lemma \ref{lem:criteria} and the fact that $Z$ and $\overline{Z}$ are skip-down-free (or skip-free to the left), we deduce
\begin{equation}\label{eq:uniformbound1} 
\mathbf{E}_n[\tau_{n_0}]=\mathbb{E}[\gamma_{n_0}|U_0=n]\leq \mathbb{E}[\overline{\gamma}_{n_0}|\overline{U}_0=n]=\overline{\mathbf{E}}_n[\overline{\tau}_{n_0}]<\overline{\mathbf{E}}_n[\overline{\tau}_{1}]<K,
\end{equation}
where  ${\gamma}_{n_0}:=\inf\{t>0: {U}_t\le n_0\}$, $\overline{\gamma}_{n_0}:=\inf\{t>0: \overline{U}_t\le n_0\}$, ${\tau}_{i}:=\inf\{t>0:{Z}_t= i\}$, $\overline{\tau}_{i}:=\inf\{t>0:\overline{Z}_t= i\}$ and $K$ is a positive constant that does not depend on $n$.

Then the process $Z$ is positive recurrent and we apply Theorem 21.14 of \cite{LPW} to conclude that there exist a unique stationary distribution for $Z$, here denoted by $\nu$. To finish the proof of the first statement, we use Markov's inequality and  (\ref{eq:uniformbound1}) to deduce 
$$
\mathbf{P}_n(\tau_{n_0}>a)<\frac{K}{a}, \qquad  \textrm{for every}\quad n\ge n_0.
$$
This allow us to compute the total variation distance between  $\mathbf{P}_n(Z_t=\cdot)$ and $\nu(\cdot)$ by conditioning on the event $\{\tau_{n_0}>t_0\}$, where $t_0\in(0,t)$. In other words
\begin{align*}
||\mathbf{P}_n(Z_t=\cdot)-\nu(\cdot)||\leq& ||\mathbf{P}_{n}(Z_t=\cdot , \tau_{n_0}\leq t_0)-\nu(\cdot)||\\&\hspace{.5 cm}+ \mathbf{P}_n(\tau_{n_0}>t_{0})||\mathbf{P}_n(Z_t=\cdot|\tau_{n_0}>t_{0})-\nu(\cdot)||\\
\le &\int_0^{t_0}||\mathbf{P}_{n_0}(Z_{t-s}=\cdot)-\nu(\cdot)||\mathbf{P}_n(\tau_{n_0}\in \ud s)+\frac{K}{t_0}.
\end{align*}
Therefore using again Theorem 21.14 of \cite{LPW}, we  observe 
\[
||\mathbf{P}_{n_0}(Z_{t-s}=\cdot)-\nu(\cdot)||\rightarrow 0 \qquad \textrm{ as }\quad  t\rightarrow \infty, 
\]  for all $s\in[0,t_0]$ and thus we conclude that
\[
\lim_{t\rightarrow \infty} \sup_{n\in \mathbb{N} }||\mathbf{P}_n(Z_t=\cdot)-\nu(\cdot)||<\frac{K}{t_0}.
\]
Since $t_0$ was taken  arbitrary, we deduce \eqref{eq:stationaryuniform}.

For  the case with catastrophes, we use an order preserving coupling.  In other words, we introduce two branching  processes  $Z$ and $\overline{Z}$ with the same parameters   $d=\overline{d}=0,$ $c=\overline{c},$ $\pi_i=\overline{\pi}_i,$ $b_i=\overline{b}_i$, for $i\in \mathbb{N}$, but the former  take catastrophes into account and the latter  has no catastrophes. Again, we refer to Section 2 of L\'opez et al. \cite{LMS} for the conditions on the intensities $(q_{i,j})_{i,j\in \mathbb{N}_0}$ in order to have  existence of an order preserving coupling which are clearly satisfied by $Z$ and $\overline{Z}$. We denote by $\overline{\mathbf{P}}_n$ for the law of $\overline{Z}$ starting at $n$. Hence, the order preserving and the upper bound in \eqref{eq:uniformbound1} implies   
 \[
 \mathbf{E}_n[\tau_{n_0}]\le \overline{\mathbf{E}}_n[\overline{\tau}_{n_0}]\le K \qquad \textrm{for} \quad n\ge n_0,
\]
where  $\overline{\tau}_{i}:=\inf\{t>0:\overline{Z}_t= i\}$ and $K$ is a positive constant that does not depend on $n$.  Then we proceed similarly as in the case with no catastrophes  in order to provide the uniform convergence to stationarity in the total variation norm. We leave the details to the reader.

In order to  prove last statement of this Theorem, we first assume  that $\mathbb{P}_x(X_\infty\in\{0,1\})<1$, where $X_\infty:=\lim_{t\rightarrow \infty}X_t$. This implies that  $\lim_{t\rightarrow \infty}\mathbb{E}_x[X_t^{n}]<\lim_{t\rightarrow \infty}\mathbb{E}_x[X_t]$, since for every $x\in (0,1)$, $x^n$ is strictly smaller than $x$. From the moment duality between $Z$ and $X$ and since $\nu$ is the invariant distribution of $Z$, we obtain
$$
\sum_{i\ge 0} x^i\nu(i)=\lim_{t\rightarrow \infty}\mathbf{E}_n[x^{Z_t}]=\lim_{t\rightarrow \infty}\mathbb{E}_x[X_t^{n}]<\lim_{t\rightarrow \infty}\mathbb{E}_x[X_t]=\lim_{t\rightarrow \infty}\mathbf{E}_1[x^{Z_t}]=\sum_{i\ge 0} x^i\mu(i),
$$
which contradicts our hypothesis. In other words, we conclude $\mathbb{P}_x(X_\infty\in\{0,1\})=1$ for any  starting point $x\in(0,1)$. Finally, we  use the left-hand side of identity \eqref{limitx} in order to  deduce that $\nu$  satisfies for any $x\in (0,1)$, 
 \[
 \sum_{i\ge 0} x^i\nu(i)=\mathbb{P}_x(X_\infty=1).
 \]
 This completes the proof.
 \end{proof}
 Next we prove  Corollary \ref{thm:criteriaX}.
\begin{proof}[Proof of Corollary \ref{thm:criteriaX}]
From the dominated convergence Theorem and the moment duality  between  $Z$ and $X$, we observe that for any $x\in (0,1)$ and $n\ge 1$, we have
$$
\mathbb{E}_x\left[\lim_{t\rightarrow\infty}X_t^n\right]=\lim_{t\rightarrow\infty}\mathbb{E}_x[X_t^n]=\lim_{t\rightarrow\infty}\mathbf{E}_n[x^{Z_t}]=\mathbf{E}_n\left[\lim_{t\rightarrow\infty}x^{Z_t}\right].
$$
From Theorem \ref{thm:criteria}, we observe that the limit of the generating function of $Z$ does not depend on the starting state $n$. This implies that all the moments of $X_\infty=\lim_{t\rightarrow\infty}X_t$ must be equal. Since the limt of $X$ must be in $[0,1]$ and the latter observation,  we deduce  that $\mathbb{P}_x(X_\infty\in\{0,1\})=1$ and thus 
$$\mathbb{P}_x\left(X_\infty=1\right)=\mathbb{E}_x\left[\lim_{t\rightarrow\infty}X_t^n\right]=\mathbf{E}_n\left[\lim_{t\rightarrow\infty}x^{Z_t}\right].$$ 
The proof is completed once we compute $\mathbf{E}_n[\lim_{t\rightarrow\infty}x^{Z_t}]$ in each of the three cases of Theorem \ref{thm:criteria}.
\end{proof}

\subsection{Proof of Theorem \ref{infinity}}
For the proof of Theorem \ref{infinity}, we need Lemma \ref{lem:criteria} and the following result.
\begin{lemma}\label{lem:criteria1}
Assume that  $\varsigma<0,$ $a=0$, $d>0$ and $\rho>0$. Then
\[
\sup_{n\in \mathbb{N} }\mathbf{E}_n\big[\tau_0\big]<\infty,
\]
where $\tau_{0}=\inf\{t>0:Z_t=0\}$.
\end{lemma}
\begin{proof} In order to prove this Lemma, we use again a coupling method. From our assumptions $Z$ has parameters  $d>0,$ $c,\pi_i,b_i$,  for $i\in \mathbb{N}$, and  $\lambda_{i,k}, $ for  $ k\in\{ 2, \ldots, i\} $. For our purposes we also consider a branching process with  interactions   $\overline{Z}=(\overline{Z}_t, t\in \mathbb{R}_{+,0})$ with parameters $\overline{d}=0$, $\overline{c}=c$, $\overline{\pi_i}=\pi_i$,  $\overline{b_i}=b_i$, for all $i\in \mathbb{N}$, and $\overline{\lambda}_{i,k}=\lambda_{i,k}, $ for   $k\in\{ 2, \ldots, i\} $. We denote by $\overline{\mathbf{P}}_n$ its law starting from $\overline{Z}_0=n$. Observe that the process $\overline{Z}$ fulfills the conditions of Theorem  \ref{thm:criteria}  and thus  there exist a unique stationary distribution $\nu$. In particular,
\[
 \lim_{t\rightarrow \infty}\overline{\mathbf{P}}_n(\overline{Z}_t=1)=\nu(\{1\})>0,
 \]
  and this is a uniform limit on the variable $n$. Thus  for every $\epsilon\in(0,\nu(\{1\}))$ there exists $t_0>0$ such that 
  \[
  \inf_{n\in \mathbb{N}}\overline{\mathbf{P}}_n(\overline{Z}_{t_0}=1)>\epsilon.
  \] 
Since the  processes only differ in the death parameter,  the order preserving coupling is straightforward. 
Let us denote such coupling by  $\{(U_t,\overline{U}_t), t\in \mathbb{R}_{+,0}\}$. In particular, we have that   $U_t$ and $\overline{U}_t$ are equal in  distribution to $Z_{t}$ and   $\overline{Z}_t$, respectively. As we said before such coupling preserves order, meaning that 
$$
U_t\leq \overline{U_t}, \qquad \textrm{for } t\in \mathbb{R}_{+,0}, \qquad \textrm{ almost surely.}
$$
On the other hand, we observe that $\mathbb{P}(U_{h}=0|U_0=1)>0$ for $h\in \mathbb{R}_{+}$. Hence, we conclude that for  $n\in \mathbb{N}$ 
\begin{align*}
\mathbf{P}_n(Z_{t_0+h}=0)=&\mathbb{P}(U_{t_0+h}=0|U_0=n)\\
>& \mathbb{P}(U_{t_0+h}=0|U_{t_0}=1)\mathbb{P}(U_{t_0}=1|U_0=n)\\
>& \mathbb{P}(U_{h}=0|U_{0}=1)\mathbb{P}(\overline{U}_{t_0}=1|\overline{U}_0=n)\\
>&\mathbf{P}_1(Z_{h}=0)\epsilon.
\end{align*}
Again, it  is important to note that the right hand side of the previous equation does not depend on $n$.  From the Markov property we deduce that
\begin{align*}
\mathbf{P}_n(\tau_0>t)\leq&\mathbf{P}_n(\tau_0>\lfloor t/(t_0+h)\rfloor t_0)\\
<&(1-\mathbf{P}_1(Z_{h}=0)\epsilon)^{\lfloor t/(t_0+h)\rfloor}.
\end{align*}
This allow us to conclude that stochastically $\tau_0\leq (t_0+h)G$ where $G$ is Geometric random variable with parameter $\mathbf{P}_1(Z_{h}=0)\epsilon$. Note that $G$ is independent of the starting condition $n.$ This leads to the conclusion  that, for every $n\in \mathbb{N}$,
 $$\mathbf{E}_n[\tau_0]<\frac{t_0+h}{\mathbf{P}_1(Z_{h}=0)\epsilon}<\infty,$$
 which completes the proof.
 \end{proof}
\begin{proof}[Proof of Theorem \ref{infinity}]
In order to prove our result, we first deduce that under our assumptions  the sequence $(\mathbf{P}_n)_{n\in \mathbb{N}}$ converges weakly in the space of probability measures on $D(\overline{\mathbb{N}}, [0, T])$, the space of Skorokhod of c\`adl\`ag functions on $[0,T]$ with values in $\overline{\mathbb{N}}$. We follow the tightness argument provided on the proof of Theorem 1 in Donelly \cite{Do}. We observe that the process does not possess instantaneous states and also that it is stochastically monotone with respect to the starting point. In other words, condition (A1) of Theorem 1 in \cite{Do} is satisfied. Moreover, under our assumptions the process does not explode and condition (A2) is also satisfied by denoting $B_n^N$ the branching process with interactions starting from $n$ and stopped at state $N$. Then tightness holds and we identify the finite marginal distributions by noticing that for $k\ge 1$, for $t_1, \ldots t_k\in \mathbb{R}_{+,0}$, and for $n_1, \dots, n_k \in \mathbb{N}$, the probabilities
 $\mathbf{P}_n(Z_{t_1}\le n_1, \cdots, Z_{t_k}\le n_k)$ are non-increasing with respect to $n$.

 Next, since $Z$ and $X$ are moment duals, for any $t\in \mathbb{R}_+$, we have 
\begin{equation}\label{comdown}
\mathbf{P}_\infty(Z_t<\infty)=\lim_{x\rightarrow 1}\lim_{n\rightarrow \infty} \mathbf{E}_n[x^{Z_t}]=\lim_{x\rightarrow 1}\lim_{n\rightarrow \infty} \mathbb{E}_x[X_t^n]=\lim_{x\rightarrow 1}\mathbb{P}_x(X_t=1), 
\end{equation}
 where in the first equality we had used that $(\mathbf{P}_n)_{n\in\mathbb{N}}$ converges weakly in the space of probability measures. In order to deduce that the previous probability equals 1, we first consider the case  $\Lambda\equiv 0$, $d=0$ and $\rho=0$. From Lemma \ref{lemma:hittingtime1}, we deduce
$$
\lim_{x\rightarrow 1}\mathbb{P}_x(X_t=1)=\lim_{x\rightarrow 1}\mathbb{P}_x(X_{T_{0,1}}=1, T_{0,1}<t)=1.
$$
By a comparison result for diffusions (see for instance Theorem 2.2 in Dawson and Li \cite{DL}), we deduce that the latter also holds true for the case $\rho>0,$ $\Lambda\equiv 0$ and $d=0$. In other words, when  $\Lambda\equiv 0$ and $d=0$,   we deduce from  \eqref{comdown} that
$$
\mathbf{P}_\infty(Z_t<\infty)=1 \qquad \textrm{for any}  \quad t\in\mathbb{R}_+,
$$   
and conclude that  the process $Z$ comes down from infinity. Finally,  the remaining cases (i.e. $d>0$ and/or catastrophes) follows from an order preserving coupling.  

Next,  we study the limit of $\mathbf{E}_n[\tau_1]$ as $n$ goes to $\infty$. If the parameters $d,c,\pi_i,b_i$,  for $i\in \mathbb{N}$,  and $\lambda_{i,k}, $ for  $k\in\{2,\ldots, i\} $ are as in Lemmas \ref{lem:criteria} and \ref{lem:criteria1}, then it is clear that $\lim_{n\to \infty}\mathbf{E}_n[\tau_1]<\infty$. If the parameters are as in Theorem \ref{thm:criteria}, then identity \eqref{eq:stationaryuniform} implies that there exists a time $t_0$ such that $\sup_{n\in\mathbb{N}}\mathbf{P}_n(Z_{t_0}=1)>\nu(1)/2.$ By the Markov property, for every $n\ge 1$, we observe  $\mathbf{P}_{n}(\tau_1>kt_0 )<\mathbf{P}_n(Z_{it_0}\neq 1, \text{ for all }i\in\{1,2,...,k\})<(1-\nu(1)/2)^k$. This implies that $\mathbf{E}_n[\tau_1]<2/\nu(1)<\infty$ for all $n\in\mathbb{N}$, in other words we have  $\lim_{n\rightarrow \infty}\mathbf{E}_n[\tau_1]<2/\nu(1)$. This completes the proof.
\end{proof}

\subsection{Proof of Proposition \ref{lem:criteriaannihilation}}

\begin{proof}[Proof of Proposition \ref{lem:criteriaannihilation}] We first deduce that  for any choice of parameter it holds that
\[
\sup_{n\in\mathbb{N}}\mathbf{E}_n[\tau_{0,1}]<\infty.
\]
Let us first assume that $a=0$.  If $d>0$ or  $d=0$ and $\pi_i=0$ for all $i\in \mathbb{N}$, then the result follows immediately from  Lemmas \ref{lem:criteria} and \ref{lem:criteria1}.  If  $d=0$ and $\pi_i>0$ for some $i\in\mathbb{N}$, then identity \eqref{eq:stationaryuniform} guarantees that for all $\epsilon<\nu(1)$ there exists $t_\epsilon>0$ such that $\inf_{n\in\mathbb{N}}\mathbf{P}_n(Z_{t_\epsilon}=1)>\epsilon.$ Thus, by the strong Markov property we conclude 
\[
\sup_{n\in\mathbb{N}}\mathbf{E}_n[\tau_{0,1}]\leq \sup_{n\in\mathbb{N}}\mathbf{E}_n[\tau_{1}]\leq\frac{t_\epsilon}{\epsilon}<\infty.
\]  
Next, we assume $a>0$.  In this case, we will use again an order preserving coupling. 
Let $Z$ be a branching process with interactions with parameters $a>0, d,$ $c,\pi_i,b_i$, and  $\lambda_{i,k}, $ for  $ k\in\{ 2, \ldots, i\} $ that represents the catastrophes. Recall that  $\mathbf{P}_n$  denotes its law starting from ${Z}_0=n$. We also  consider another  branching process with  interactions  $\overline{Z}=(\overline{Z}_t, t\in \mathbb{R}_{+,0})$, with parameters $\overline{a}=0, \overline{d}=d$, $\overline{\pi_i}=\pi_i, \overline{c}=c+a, \overline{b_i}=b_i$ for all $i\in \mathbb{N}$ and $\overline{\lambda}_{i,k}, $ for  $ k\in\{ 2, \ldots, i\} $.  We denote by $\overline{\mathbf{P}}_n$  its law starting from $\overline{Z}_0=n$.  Since the  only parameters in which the processes $Z$ and $\overline{Z}$ differ are the annihilation and  competition,  the coupling is straightforward (every time there is an annihilation event in $Z$ there is a competition event in $\overline{Z}$). 

Let us denote by  $\{(U_t,\overline{U}_t), t\in \mathbb{R}_{+,0}\}$ such  coupling and observe 
\[
U_t\leq \overline{U_t}, \qquad \textrm{for } t\in\mathbb{R}_{+,0}, \qquad \textrm{ almost surely.}
\]
Since $\overline{U}$ has no annihilation, we conclude by the previous argument that
$$\sup_{n\in\mathbb{N}}\mathbf{E}_n[\tau_{0,1}]<\sup_{n\in\mathbb{N}}\mathbf{E}_n[\overline{\tau}_{0,1}]<\infty,$$
where $\overline{\tau}_{0,1}$ denotes the first hitting time of the states $\{0,1\}$ of the process $\overline{Z}$.
Whenever $d>0$, we have
$$\sup_{n\in\mathbb{N}}\mathbf{E}_n[\tau_{0}]<\sup_{n\in\mathbb{N}}\mathbf{E}_n[\overline{\tau}_{0}]<\infty.$$
This prove part (ii).

From the first part of the proof, we know that $Z$ visits $\{0,1\}$. Observe that if $Z$ visits the state $\{0\}$ it will get absorbed, but if it visits the state $\{1\}$ it will get absorbed if and only if $d=0$ and $\pi_i=0$ for all ${i\in\mathbb{N}}$. 
The proofs of parts (a) and (b) relies in verifying whether  the states $\{0\}$ and $\{1\}$  are absorbing states or not.

For part (a), the assumptions guarantee that for every $n\in\mathbb{N}\setminus\{1\}$ there exist some number $i\in\mathbb{Z}$ such that $Z$ goes from $n$ to $n+2i-1$ at a positive rate. Since $a>0$, then $Z$ goes from $n$ to $n-2$ at a positive rate. In other words,  the states  $\{0\}$ and $\{1\}$ are accessible.

On the other hand, we deduce that  $\{1\}$ is an absorbing state if and only if $\sum_{i\in\mathbb{N}}\pi_i=0$.  If $\{1\}$ is an absorbing state, then
 from the first part of the proof, we deduce $\sup_{{n\in\mathbb{N}}}\mathbf{E}_n[\tau_{0,1}]<\infty.$ If $\{1\}$ is not absorbing, then $\{0\}$ is accessible from $\{1\}$, which implies that there exists $\epsilon>0$ and $t_\epsilon\in\mathbb{R}_+$ such that 
$$
\mathbf{P}_1(Z_{t_\epsilon}=0)>\epsilon.
$$
Using the strong Markov property we conclude
$$
\sup_{n\geq1}\mathbf{E}_n[\tau_{0}]<\frac{1}{\epsilon}\left(\sup_{n\geq1}\mathbf{E}_n[\tau_{0,1}]+t_\epsilon\right)<\infty.
$$
This completes the proof of part (a).

For part (b), we observe that $Z$ only makes jumps of even size, then the support of $\mathbf{P}_n(Z_t=\cdot)$ are the even or the odd numbers depending the value $n$. If $n$ is even, then  $\{1\}$ is not accessible and $\{0\}$ is accessible. In other words,  $\tau_{0,1}=\tau_{0}$ and  by the first part of our proof,   we conclude  $\sup_{{n\in\mathbb{N}}}\mathbf{E}_n[\tau_{0}]<\infty$. Similarly, if $n$ is odd then $\{1\}$ is accessible but not the state $\{0\}$. In this case   $\tau_{0,1}=\tau_{1}$ and again from  the first part of the proof we conclude $\sup_{{n\in\mathbb{N}}}\mathbf{E}_n[\tau_{1}]<\infty$.

On the other hand  under our assumptions, we deduce that $\{1\}$ is an absorbing state if and only if $\sum_{i\in\mathbb{N}}\pi_i=0$. If $\sum_{{i\in\mathbb{N}}}\pi_i>0$, we have that  $\{1\}$ is positive recurrent. The uniform convergence to the stationary distribution in this case follows form similar arguments as those used in  Theorem \ref{thm:criteria}, part $(ii)$. We leave the details to the reader.
\end{proof}

\section*{Acknowledgments}
All authors would like to thank Gabriel Berzunza for many useful discussions as well as two
anonymous referees whose careful reading  led to significant
improvements.
AGC acknowledges support from UNAM through the grant PAPIIT IA100419 and from the German Research Foundation through the Priority Programme 1590 \textit{Probabilistic Structures in Evolution.} JCP and JLP also acknowledge support from the Royal Society and CONACyT-MEXICO.

\end{document}